\documentclass[reqno]{amsart}
\usepackage[T1]{fontenc}
\usepackage[english]{babel}
\usepackage{amsmath}
\usepackage{amsthm}
\usepackage{amssymb}
\usepackage[utf8]{inputenc}
\usepackage[T2A]{fontenc}
\usepackage{bbm,wasysym}
\usepackage{hyperref}
\usepackage{bbold}
\usepackage{graphicx}
\usepackage{xcolor}
\usepackage[numbers,square]{natbib}

\newcommand{\stirlingfirstkind}[2]{%
  \genfrac{[}{]}{0pt}{}{#1}{#2}}

\newcommand{\stirling}[2]{%
  \stirlingfirstkind{#1}{#2}}

\newcommand{\stirlingb}[2]{%
  \mathrm{B}\mkern-2mu\stirlingfirstkind{#1}{#2}}

\newcommand{\pwend}[2]{%
  p_{\mathrm{Wend}}\mkern-2mu
  \genfrac{(}{)}{0pt}{}{#1}{#2}}

\newcommand{\pwalk}[2]{%
  p_{\mathrm{walk}}\mkern-2mu
  \genfrac{(}{)}{0pt}{}{#1}{#2}}

\newcommand{\pbridge}[2]{%
  p_{\mathrm{br}}\mkern-2mu
  \genfrac{(}{)}{0pt}{}{#1}{#2}}

\newcommand{\qwend}[2]{%
  q_{\mathrm{Wend}}\mkern-2mu
  \genfrac{(}{)}{0pt}{}{#1}{#2}}

\newcommand{\qwalk}[2]{%
  q_{\mathrm{walk}}\mkern-2mu
  \genfrac{(}{)}{0pt}{}{#1}{#2}}

\newcommand{\qbridge}[2]{%
  q_{\mathrm{br}}\mkern-2mu
  \genfrac{(}{)}{0pt}{}{#1}{#2}}

\DeclareMathOperator{\sgn}{sgn}

\DeclareMathOperator{\Conv}{Conv}

\DeclareMathOperator{\relint}{relint}
\DeclareMathOperator{\Int}{Int}

\newcommand{\pr}{\mathbb{P}}
\newcommand{\R}{\mathbb{R}}
\newcommand{\eps}{\varepsilon}

\theoremstyle{plain}
\newtheorem{theorem}{Theorem}
\newtheorem{lemma}[theorem]{Lemma}
\newtheorem{proposition}[theorem]{Proposition}
\newtheorem{corollary}[theorem]{Corollary}

\theoremstyle{definition}

\theoremstyle{remark}
\newtheorem{remark}[theorem]{Remark}

\title[Random Convex Hulls via Wall Crossing]{Absorption Probabilities for Random Convex Hulls: Distribution-Freeness via the Wall-Crossing Method}
\date{}

\begin{document}

\author[Kabluchko]{Zakhar Kabluchko}
\address{Institut f\"ur Mathematische Stochastik,
	Universität M\"unster,  Germany}
\email{zakhar.kabluchko@uni-muenster.de}

\author[Tarasov]{Alexander Tarasov}
\address{Universit\"at Hildesheim, Germany}
\email{tarasov@uni-hildesheim.de}



\begin{abstract}
    We consider the probability that the convex hull of the first $n$ partial sums of a
    $d$-dimensional random walk contains the origin. Under symmetric
    exchangeability of the increments and a general-position assumption, this absorption probability is distribution-free and admits an explicit formula, previously obtained by Kabluchko, Vysotsky and Zaporozhets [Geom. Funct. Anal. 27 (2017)] using characteristic polynomials of hyperplane arrangements. We give a different proof, based on a wall-crossing method
    which we develop here. Starting from a deterministic configuration of
    increments, we count the signed permutations for which the convex hull of
    the corresponding partial sums contains the origin and show that this count
    remains unchanged under generic deformations of the increments, and hence
    is the same for all configurations outside a natural exceptional set
of measure zero. Evaluating the invariant
    at a single well-chosen configuration reduces the remaining calculation to
    the enumeration of permutation records combined with Wendel's theorem. Our method also reproves Wendel's theorem on convex hulls of random points with a
sign-flip-invariant joint distribution and, in dimension one,
Sparre Andersen's theorem.
    Finally, we derive new probabilistic representations and recurrence
    relations for the absorption probabilities of random-walk convex hulls and
    their random-bridge analogues.
\end{abstract}

\maketitle

\section{Introduction}\label{sec:intro}

\subsection{Statement of results}
A classical distribution-free theorem of Sparre
Andersen~\cite{SA49} asserts that if
$\xi_1,\dots,\xi_n$ are i.i.d.\ real-valued random variables with a
continuous distribution symmetric about the origin, then their partial sums
\[
    S_k:=\xi_1+\cdots+\xi_k,
    \qquad 1\leq k\leq n,
\]
satisfy
\begin{align}
    \pr(S_1>0,\dots,S_n>0)
    &=
    \frac{(2n-1)!!}{2^n n!}
    =
    \frac{1}{2^{2n}}\binom{2n}{n}.
    \label{eq:sparreandersen}
\end{align}
Several proofs of Sparre Andersen's formula are available; see, for example, Feller~\cite[Chapter~XII]{FellerVol2}, Spitzer~\cite{Spitzer1956}, Durrett~\cite[Theorem~4.3.7]{Durrett2010}, Majumdar~\cite[Section~3.2]{Majumdar2010}, Pinsky~\cite[Section~12]{Pinsky2021}, Berger and Béthencourt~\cite{BergerBethencourt2025}.

By symmetry, the same formula holds for the probability that all partial
sums are negative. Let $\Conv(\cdot)$ denote the convex hull operator. In one dimension,
$
0\notin\Conv(S_1,\dots,S_n)
$
if and only if the partial sums are either all positive or all negative.
Since these two events are disjoint, it follows that
\[
    \pr\bigl(0\notin\Conv(S_1,\dots,S_n)\bigr)
    =
    \frac{2}{2^{2n}}\binom{2n}{n}.
\]

This observation suggests that a natural $d$-dimensional extension of
Sparre Andersen's theorem should concern the probability that the convex
hull of a $d$-dimensional random walk contains (or avoids) the origin. This
problem was studied in dimension $d=2$
in~\cite{vysotsky_zaporozhets_convex_hulls_TAMS} and in arbitrary dimension
in~\cite{KVZ17_GAFA}.

The following $d$-dimensional version of Sparre Andersen's theorem was
proved in~\cite{KVZ17_GAFA}.

\begin{theorem}[Absorption probability for a random walk] \label{thm:1}
Let $d\geq 1$ and $n\geq d+1$. Let $\xi_1,\dots,\xi_n$ be (possibly dependent) random $d$-dimensional vectors with partial sums
$$
S_i = \xi_1 + \dots + \xi_i,\quad  1\leq i\leq n,\quad  S_0=0.
$$
We impose the following conditions on the joint distribution of the tuple $(\xi_1,\dots, \xi_n)$:
\begin{itemize}
\item [(i)] \emph{Symmetric exchangeability:} For every permutation $\sigma$ of $\{1,\dots,n\}$ and every choice of
signs $\eps_1,\dots,\eps_n\in\{+1,-1\}$, we have the equality in
distribution
    $$
    (\xi_1,\dots,\xi_n) \stackrel{d}{=}(\eps_1 \xi_{\sigma(1)}, \dots, \eps_n \xi_{\sigma(n)}).
    $$
\item[(ii)] \emph{General position of partial sums:}
For every $1\leq i_1 < \dots < i_d\leq n$, the probability that the vectors $S_{i_1}, \dots,S_{i_d}$ are linearly dependent is $0$.
\end{itemize}

Then
\begin{align}
    \pwalk{n}{d}
    &:=
    \pr\big(0 \in \Conv(S_1, S_2, \dots, S_n)\big)
    =
    \frac{2}{2^{n}n!}
    \sum_{r\geq 0}
    \stirlingb{n}{d+1+2r},
    \label{eq:def-p-walk}
    \\
        \qwalk{n}{d}
    &:=
    \pr\big(0 \notin \Conv(S_1, S_2, \dots, S_n)\big)
    =
    \frac{2}{2^{n}n!}
    \sum_{r\geq 0}
    \stirlingb{n}{d-1-2r}.
    \label{eq:def-q-walk}
\end{align}
Here, the numbers $\stirlingb{n}{k}$ with $n\in \{1,2,\dots\}$ and $k\in \mathbb Z$ are the $B$-Stirling numbers of the first kind defined by the generating function
\begin{align*}
    (t+1)(t+3) \cdots (t+2n-1)
    = \sum_{k\in \mathbb Z} \stirlingb{n}{k}\, t^k.
\end{align*}
In particular, $\stirlingb{n}{k}= 0$ for $n\in \{1,2,\dots\}$ and $k\notin\{0,\dots,n\}$, so that the sums appearing above terminate after finitely many nonzero terms.
\end{theorem}

Another seemingly unrelated distribution-free result on absorption probabilities for random convex hulls is Wendel's classical theorem~\cite{Wendel62}; see also~\cite[Theorem~8.2.1]{schneider_weil_book}.

\begin{theorem}[Wendel's formula]
\label{thm:wendel}
Let $d\geq 1$ and $n\geq d+1$.  Let  $\xi_1,\dots, \xi_n$ be (possibly dependent) random vectors in $\R^d$ such that the following conditions hold:
\begin{itemize}
\item[(i)] \emph{Sign-flip invariance}:  For every vector of signs $(\eps_1,\dots, \eps_n)\in \{+1,-1\}^n$, we have the distributional equality
$$
(\eps_1 \xi_1,\dots, \eps_n \xi_n)
\stackrel{d}{=} (\xi_1,\dots, \xi_n).
$$
\item[(ii)] \emph{General position}:  For every $1\leq i_1 < \dots < i_d\leq n$, the probability that the vectors $\xi_{i_1}, \dots,\xi_{i_d}$ are linearly dependent is $0$.
\end{itemize}
Then
\begin{align}
    \pwend{n}{d}
    &:=
    \pr[0\in \Conv(\xi_1,\dots, \xi_n)]
    =
    \frac{1}{2^{n-1}}
    \sum_{j=d}^{n-1}
    \binom{n-1}{j},
    \label{eq:def-p-Wend}
    \\
    \qwend{n}{d}
    &:=
    \pr[0\notin \Conv(\xi_1,\dots, \xi_n)]
    =
    \frac{1}{2^{n-1}}
    \sum_{j=0}^{d-1}
    \binom{n-1}{j}.
    \label{eq:def-q-Wend}
\end{align}
\end{theorem}

The proofs of Theorems~\ref{thm:1} and~\ref{thm:wendel} given in~\cite{KVZ17_GAFA} are based on the theory of hyperplane arrangements and their characteristic polynomials. For Theorem~\ref{thm:1}, the absorption event is interpreted in terms of intersections between a linear subspace of codimension $d$ and the Weyl chambers of type $B_n$ in $\R^n$. It is then shown that every generic linear subspace of fixed codimension intersects the same number of Weyl chambers,  and this number is computed using Whitney's formula for the
characteristic polynomial together with Zaslavsky's theorem expressing the number of chambers through
that polynomial. In particular, the numbers $\stirlingb{n}{k}$ are, up to alternating signs, the coefficients of the characteristic polynomial of the reflection arrangement associated with the group $\mathfrak B_n$. This group acts on $\R^n$ by permuting the coordinates and changing the signs of an arbitrary subset of them. The symmetric-exchangeability assumption in Theorem~\ref{thm:1} is precisely invariance under this action.
A similar method was used in~\cite{KVZ17_GAFA} to reprove Wendel's theorem and to establish a random-bridge analogue of Theorem~\ref{thm:1}, which will be stated in Section~\ref{sec:bridge}. The relevant reflection groups are $(\mathbb Z/2\mathbb Z)^n$ in Wendel's setting, acting on $\R^n$ by coordinate sign changes, and the symmetric group $\mathfrak S_n$ in the bridge setting, acting by coordinate permutations. Moreover,~\cite{KVZ17_Advances} established a result unifying these theorems: it considers an arbitrary number of random walks and bridges in $\R^d$ and computes the probability that their joint convex hull contains the origin. A single walk recovers Theorem~\ref{thm:1}, while $n$ walks of length $1$ recover Wendel's formula. The reflection group underlying the general result is a direct product of groups of the form $\mathfrak B_m$ and $\mathfrak S_m$.

\medskip
The purpose of the present paper is to give alternative proofs of the
absorption formulas for Wendel's model, random walks, and random bridges
that entirely avoid the theory of hyperplane arrangements.
Our
approach is based on a deformation argument, which we call the
\emph{wall-crossing method}.

\subsection{Outline of the method.}
We first illustrate the wall-crossing method in the setting of Theorem~\ref{thm:1}.
The method is essentially deterministic. We start with
vectors $x_1,\dots,x_n\in\R^d$, which play the role of the increments, and
arrange them as the columns of a matrix
\[
    X=(x_1\,|\,\dots\,|\,x_n)\in\R^{d\times n}.
\]
We then consider all $2^n n!$ signed permutations of these vectors. More
precisely, for every permutation $\sigma$ of $\{1,\dots,n\}$ and every
vector of signs $\eps=(\eps_1,\dots,\eps_n)\in\{\pm1\}^n$, we form the
polytope
\[
    \Conv\left(
        \sum_{i=1}^k \eps_i x_{\sigma(i)}
        :1\leq k\leq n
    \right)
    \subseteq\R^d.
\]
This gives a collection of $2^n n!$ convex polytopes associated
with $X$.  (Here and below, these polytopes are counted with multiplicity, according to the signed permutations that produce them.) Our aim is to show that, for every $X$ outside an exceptional set
$\mathcal D\subseteq \R^{d\times n}$ of Lebesgue measure zero, exactly
\[
    2\sum_{r\geq0}\stirlingb{n}{d+1+2r}
\]
of these polytopes contain the origin. A precise description of
$\mathcal D$ will be given below. Dividing this number by $2^n n!$ yields
the absorption probability in Theorem~\ref{thm:1}. In fact, this
deterministic statement is equivalent to that theorem.

\smallskip
The proof consists of two main parts.

\smallskip
\emph{Part 1: Constancy outside the exceptional set.}
We first show that the number of polytopes containing the origin is
independent of $X$ outside $\mathcal D$. To this end, we let $X$ vary
continuously in $\R^{d\times n}$. As long as the origin does not lie on the
boundary of any of the $2^n n!$ polytopes, the number of polytopes
containing the origin remains locally constant.

The configurations for which the origin lies on the boundary of at least
one polytope form a finite union of algebraic hypersurfaces in
$\R^{d\times n}$, which we call \emph{walls}. It therefore remains to
analyze what happens when $X$ crosses a wall. It suffices to consider a generic wall crossing at which, for each signed
permutation, at most one relevant $d$-tuple of partial sums becomes
linearly dependent. Configurations violating this condition have codimension at
least $2$ and can be avoided
by a suitable choice of path.  The key step is
the following assertion.
\begin{quote}
Let $X_0$ be a generic wall-crossing point, and consider only the polytopes whose
boundaries contain the origin. When $X_0$ is perturbed to  a nearby configuration $X$
off the wall, these polytopes split into two classes of
\textbf{equal cardinality}: those in the first class contain the origin,
whereas those in the second do not.  Thus, whenever one polytope acquires the origin, another loses it.
\end{quote}
Consequently, the number of polytopes containing the origin is the same
before and after a simple wall crossing, although its value at the
crossing point itself may be different.

\smallskip
\emph{Part 2: Computation of the constant.}
It remains to evaluate this number for one suitably chosen configuration
of increments outside the exceptional set. This choice is
problem-specific. For random walks, we take generic unit vectors
$u_1,\dots,u_n\in\R^d$ and set
$
    x_k:=R^{k-1}u_k
$
for sufficiently large $R>0$. We then prove that
\[
    0\in\Conv\left(
        \sum_{i=1}^k
        \eps_i R^{\sigma(i)-1}u_{\sigma(i)}
        :1\leq k\leq n
    \right)
    \quad\Longleftrightarrow\quad
    0\in\Conv\left(
        \eps_i u_{\sigma(i)}
        :i\in\mathrm{Rec}(\sigma)
    \right),
\]
where
$$
    \mathrm{Rec}(\sigma)
    :=
    \left\{
        k\in\{1,\dots,n\}:
        \sigma(k)>\sigma(i)\text{ for every }i<k
    \right\}
$$
is the set of record positions of $\sigma$.
The number of records of a uniform random permutation has the same
distribution as the number $K_n$ of its cycles. In particular,
\[
    \pr(K_n=k)
    =
    \frac{1}{n!}\stirling{n}{k},
    \qquad 1\leq k\leq n,
\]
where $\stirling{n}{k}$ are the unsigned Stirling numbers of the first kind,
defined by
\[
    \sum_{k=0}^n\stirling{n}{k}t^k
    =
    t(t+1)\cdots(t+n-1).
\]
Conditional on a permutation with $k$ records, Wendel's formula gives
the probability that the convex hull of the corresponding signed record
vectors contains the origin. Averaging over the number of records yields
\[
    \pwalk{n}{d}
    =
    \mathbb E\left[\pwend{K_n}{d}\right]
    =
    \frac{1}{n!}
    \sum_{k=d+1}^n
    \stirling{n}{k}\,
    \pwend{k}{d},
\]
with the natural convention $\pwend{k}{d}:=0$ for $1\leq k\leq d$.
A straightforward coefficient calculation transforms this identity into
the explicit formula~\eqref{eq:def-p-walk}. The arguments developed below
also yield the recurrence
\begin{equation}
    \pwalk{n}{d}
    =
    \frac{2n-1}{2n}\pwalk{n-1}{d}
    +
    \frac{1}{2n}\pwalk{n-1}{d-1}
    \label{eq:p-walk-recurrence-intro}
\end{equation}
and the probabilistic representation
\begin{equation}
    \pwalk{n}{d}
    =
    \pr\left(\sum_{m=2}^n U_m\geq d\right),
    \label{eq:p-walk-bernoulli-intro}
\end{equation}
where $U_2,U_3,\dots$ are independent random variables satisfying
$
U_m\sim\operatorname{Bern}\left(\frac{1}{2m}\right)$.

\medskip
The wall-crossing method is flexible enough to apply to many other
distribution-free problems. In particular, we expect that it can be
extended to convex hulls of several random walks and bridges, which would
provide an alternative proof of a result from~\cite{KVZ17_Advances}, as
well as to expected face numbers and positive hulls of random walks.
Analogous results in which partial sums are replaced by successive
differences were obtained
in~\cite{godland_kabluchko_positive_hulls_rand_walks,
godland_kabluchko_conical_tess_weyl_chambers}. To avoid excessive
technicalities, we illustrate the method through several representative
examples rather than pursuing its most general formulations.

 \subsection{Organization of the paper}
In Section~\ref{sec:Sparre-Andersen}, we illustrate the wall-crossing
method by reproving the classical formula of Sparre Andersen.
Section~\ref{sec:wendel} applies the method to Wendel's theorem, while
Section~\ref{sec:convex_hulls_random_walks} treats random-walk convex hulls
and proves Theorem~\ref{thm:1}. In Section~\ref{sec:bridge}, we adapt the
method to random bridges. Section~\ref{sec:insertion} gives a unified geometric
derivation of recurrence relations for absorption probabilities by an insertion argument and explains their probabilistic interpretation.
Section~\ref{sec:absorption_properties} derives  probabilistic
representations and recurrences directly  from the explicit formulas.

Appendix~\ref{convex_geometry} collects standard facts from convex
geometry that are used throughout the paper. Appendix~\ref{sec:paths_existence} contains the auxiliary
results on the existence of paths avoiding sets of codimension at least
$2$.

\section{Sparre Andersen's formula}
\label{sec:Sparre-Andersen}

In this section we illustrate the method in its simplest, one-dimensional
case, reproving Sparre Andersen's formula~\eqref{eq:sparreandersen}.

We begin by stating a combinatorial version of Sparre Andersen's formula. Fix an integer $n\geq1$. For a vector $x= (x_1,\dots, x_n)\in \R^n$ we denote its partial sums by
$$
s_k(x) := x_1 + x_2 + \dots + x_k,
\qquad 1 \le k \le n.
$$
Consider the closed polyhedral cone
\begin{align*}
C_{\mathrm{SA}} := \{x \in \R^n \mid s_1(x) \ge 0,
        s_2(x) \ge 0,\ \dots,\ s_n(x) \ge 0\}.
\end{align*}

Let $\mathfrak{B}_n$ be the group of signed permutations of the set $[n]:= \{1,\dots,n\}$ acting on $\R^n$ by permuting the coordinates of a vector and  multiplying
an arbitrary subset of the coordinates by $-1$. It is the symmetry group of the cube $[-1,1]^n$ and has $2^n n!$ elements. We parametrize each  $g \in \mathfrak{B}_n$ by
a permutation $\sigma$ of $\{1,\dots, n\}$ and a vector of signs $\varepsilon \in \{\pm 1\}^n$,
so that $g:\R^n \to \R^n$ acts by the orthogonal transformation
$$
gx = \big(\varepsilon_1 x_{\sigma(1)}, \dots,
\varepsilon_n x_{\sigma(n)}\big), \qquad x\in \R^n.
$$
In this section, we identify the elements of $\mathfrak{B}_n$ with the corresponding orthogonal transformations of $\R^n$.

For a point $x\in \R^n$ let $\mathcal N(x)$ denote the set of all signed permutations that map $x$ to the cone $C_{\mathrm{SA}}$, that is
\begin{align*}
    \mathcal N (x) := \{g \in \mathfrak{B}_n \mid gx \in C_{\mathrm{SA}}\}.
\end{align*}
Sparre Andersen's formula states that the size of the set $\mathcal N(x)$ is the same for all $x$ outside some ``exceptional set'' which we are now going to describe.  For $g \in \mathfrak{B}_n$ and
$m \in \{1, \dots, n\}$ consider the hyperplane
\begin{align*}
    L_{g,m} := \{x \in \R^n \mid s_{m}(gx) = 0\}.
\end{align*}
The hyperplanes $L_{g,m}$ are in general not distinct. Let $\mathcal{L}$ be the finite set of all such hyperplanes listed without repetitions. A point $x \in \R^n$ is \emph{regular} if
it lies on no hyperplane of $\mathcal{L}$, i.e.\ if $x\in \R^n\setminus \bigcup \mathcal L$. Equivalently, a point $x\in \R^n$ is regular if $\sum_{i\in I} \eps_i x_{i} \neq 0$ for every nonempty subset $I\subseteq [n]$ and for every vector of signs $(\eps_i)_{i\in I}\in \{\pm 1\}^I$.

\begin{theorem}[Sparre Andersen's formula: Combinatorial version]\label{theo:1d_sparre_andersen}
    For every regular point
    $x \in \R^n$, the cardinality of the set  $\mathcal{N}(x)$ is given by
    $$
    |\mathcal{N}(x)| = (2n-1)!! = \frac{(2n)!}{2^{n}n!}.
    $$
    In particular, this cardinality is the same for all regular points $x$.
\end{theorem}

\begin{proof}
In the first part of the proof we show that $|\mathcal N(x)| = |\mathcal N(y)|$ for all regular points $x,y \in \R^n$. To this end, we connect $x$ and $y$ by a suitable path and show that $|\mathcal N(x)|$ is constant along this path, excluding finitely many exceptional times. This shows that $|\mathcal N(x)|$ is constant on the set of regular points. In the second part of the proof, we identify the value of the constant by evaluating $|\mathcal N(x)|$ at a suitably chosen regular point $x$.

\medskip
\emph{Constancy along a path.} Let $x, y\in \R^n$ be any two regular points.
There is a piecewise linear path
$\gamma : [-1,1] \to \R^n$ such that
\begin{itemize}
\item $\gamma$ connects $x$ to $y$, that is $\gamma(-1) =x$ and $\gamma(1) = y$;
\item $\gamma(t)\in\bigcup\mathcal L$ only at finitely many exceptional
times $t_1,\dots,t_L$ with  $-1 < t_1 < \dots < t_L< 1$;
\item  $\gamma$ avoids double degeneracies, i.e.\ for every $i=1,\dots, L$, the point  $\gamma(t_i)$ belongs to \emph{exactly one} hyperplane in $\mathcal{L}$.
\end{itemize}

Heuristically, this follows because the linear subspaces $H'\cap H''$ with $H',H''\in \mathcal L$, $H'\neq H''$, have codimension at least $2$ and can therefore be avoided by a small perturbation of the path. A complete proof is given  in Lemma \ref{lem:path_for_sparre_andersen}.

Recall that $\mathcal{N}(\gamma(t)) = \{g \in \mathfrak{B}_n \mid g(\gamma(t)) \in C_{\mathrm{SA}}\}$. We study the behavior of the function   $t\mapsto |\mathcal{N}(\gamma(t))|$ when $t$ changes from $-1$ to $1$. Fix some $g\in \mathfrak{B}_n$. Let first $L\geq 1$. On every interval $[-1, t_1), (t_1,t_2), \dots, (t_L, 1]$, we have $g(\gamma(t))\notin \bigcup \mathcal L$ by the properties of $\gamma$ and by the $\mathfrak{B}_n$-invariance of $\bigcup \mathcal L$. By the definition of $\mathcal L$, none of the functions
$t\mapsto s_m(g(\gamma(t)))$, $1\leq m\leq n$, vanishes on any of
these intervals. Hence each of these functions has a constant sign
there.  So $|\mathcal{N}(\gamma(t))|$ is constant on the intervals $[-1, t_1), (t_1,t_2), \dots, (t_L, 1]$. (We assumed $L\geq 1$; for $L=0$ the same argument shows that $|\mathcal{N}(\gamma(t))|$ is constant on $[-1,1]$, and the conclusion follows.)
It remains to verify that $|\mathcal{N}(\gamma(t))|$  does not change when $t$ crosses an exceptional time $t_\ell$ (although the value exactly at $t_\ell$ may be different).

Fix an exceptional time $t_\ell$ with $\ell\in \{1,\dots, L\}$. After reparametrization of the path $\gamma$ we may assume that $t_\ell= 0$ --- this is done just for convenience of notation.  Our aim is to show that for all sufficiently small $t\neq 0$,  the function $t\mapsto |\mathcal N (\gamma(t))|$ stays constant.

First, let $g\in \mathfrak{B}_n$ be such that $g(\gamma(0)) \notin C_{\mathrm{SA}}$. Since $C_{\mathrm{SA}}$ is closed and $\gamma$ is continuous, we have $g(\gamma(t)) \notin C_{\mathrm{SA}}$ for all sufficiently small $t$. It follows that such elements $g$ contribute $0$ to $|\mathcal{N}(\gamma(t))|$.  Next, let  $g\in \mathfrak{B}_n$ be such that $g(\gamma(0))$ belongs to the interior of $C_{\mathrm{SA}}$. Then, again by continuity, $g(\gamma(t))\in C_{\mathrm{SA}}$ for all sufficiently small $t$.  Such elements $g$ contribute $1$ to $|\mathcal{N}(\gamma(t))|$. Let $\partial C_{\mathrm{SA}}$ be the boundary of $C_{\mathrm{SA}}$ and write
$$
|\mathcal N(\gamma(t))| = \sum_{g\in \mathfrak{B}_n: g(\gamma(0))\notin \partial C_{\mathrm{SA}}} \mathbb{1}_{\{g(\gamma(t)) \in C_{\mathrm{SA}}\}} + \sum_{g\in \mathfrak{B}_n: g(\gamma(0))\in \partial C_{\mathrm{SA}}} \mathbb{1}_{\{g(\gamma(t)) \in C_{\mathrm{SA}}\}}.
$$
We have shown that the first sum on the right-hand side stays constant provided $t$ is sufficiently small.

It remains to consider  $g\in \mathfrak{B}_n$ for which $g(\gamma(0))$ belongs to $\partial C_{\mathrm{SA}}$.    This means that $s_{k}(g(\gamma(0)))\geq 0$ for all  $k\in \{1,\dots, n\}$ and $s_{m}(g(\gamma(0))) = 0$ for at least one $m\in \{1,\dots, n\}$. Since $\gamma(0)$ belongs  to  a \emph{unique} hyperplane of $\mathcal{L}$, the index $m$ with this property is unique and we write $m= m(g)$. So, let
\begin{multline*}
\mathcal N_{\partial}  = \{g\in \mathfrak{B}_n:
\exists m = m(g)\in \{1,\dots, n\}:
\\
s_{m}(g(\gamma(0))) = 0 \text{ and }
s_{k}(g(\gamma(0))) > 0 \text{ for } k \neq m\}.
\end{multline*}

To complete the proof,  it suffices to show that for all $t\neq 0$ sufficiently close to $0$, one has
\begin{equation}\label{eq:sparre_proof_sum_over_M_ell_half}
\sum_{g\in \mathcal N_\partial} \mathbb{1}_{\{g(\gamma(t)) \in C_{\mathrm{SA}}\}} = \frac 12 |\mathcal N_\partial|.
\end{equation}

To prove~\eqref{eq:sparre_proof_sum_over_M_ell_half}, we shall construct a bijection $\iota: \mathcal N_\partial \to \mathcal N_\partial$ such that  $\mathbb{1}_{\{g(\gamma(t)) \in C_{\mathrm{SA}}\}}
    + \mathbb{1}_{\{\iota(g)\gamma(t) \in C_{\mathrm{SA}}\}} = 1$ for all $g\in \mathcal N_\partial$ and all  sufficiently small $t\neq 0$.
For each $m\in\{1,\dots,n\}$,  let $\rho_m \in \mathfrak{B}_n$ be a signed permutation that reverses
the order and flips the signs of the first $m$ coordinates, that is $\rho_m:\R^n \to\R^n$ is a linear map given by
$$
\rho_m (x_{1},\dots, x_{n}) =
(-x_{m}, -x_{m-1},\dots, -x_{2}, -x_{1}, x_{m+1},x_{m+2},\dots, x_{n}).
$$
It follows from this definition that
\begin{equation}\label{eq:sparre_proof_rho_m_changes_sign_m_partial_sum}
s_{m}(\rho_m x) = -\,s_{m}(x).
\end{equation}
Also,  $\rho_m \rho_m = \mathrm{id}$, i.e.\ $\rho_m$ is an involution.
A telescoping computation gives, for all $x\in \R^n$,
\begin{equation} \label{eq:partial_sum_after_rho}
    s_{j}(\rho_m x)
    =
    \begin{cases}
    s_{m-j}(x) - s_{m}(x), &j=1,\dots, m,
    \\
     s_{j}(x) - 2\, s_{m}(x), &j =  m,\dots, n,
     \end{cases}
\end{equation}
where we set  $s_{0}(x) := 0$. Note that the two cases agree at $j=m$.
For $g\in \mathcal N_\partial$ define $\iota(g) := \rho_{m(g)}\, g$. It follows from~\eqref{eq:sparre_proof_rho_m_changes_sign_m_partial_sum} and \eqref{eq:partial_sum_after_rho} that the $m$-th partial sum of $\iota(g)\gamma(0)$ vanishes, while all other partial sums are positive. So, $\iota$ maps $\mathcal N_\partial$ to itself and, moreover, $m(\iota(g)) = m(g)$. The latter property, together with $\rho_m \rho_m =\mathrm{id}$, implies that $\iota \iota = \mathrm{id}$, i.e.\ $\iota: \mathcal N_\partial \to \mathcal N_\partial$ is an involution and, consequently, a bijection.

For $g\in \mathcal N_\partial$ and $m=m(g)$, the definition of $\mathcal N_\partial$ implies that the $m$-th  partial sum of $g(\gamma(0))$ vanishes, while all other partial sums are strictly positive. The same claim applies to the partial sums of $\iota(g)(\gamma(0))$ since $\iota(g)\in \mathcal N_\partial$ and $m(g) = m(\iota(g))$. Moreover, it follows from~\eqref{eq:sparre_proof_rho_m_changes_sign_m_partial_sum} that  for every $t$ we have
$$
s_{m}(g(\gamma(t))) +  s_{m}(\iota(g)(\gamma(t))) = 0.
$$
For sufficiently small $t\neq 0$, exactly one of these numbers is strictly positive, whereas the other is strictly negative. Moreover, by continuity, the $k$-th partial sums of $g(\gamma(t))$ and $\iota(g)(\gamma(t))$ remain positive for all $k\neq m$ and all sufficiently small $t$.  It follows that for sufficiently small $t\neq 0$,  exactly one of the points $g(\gamma(t))$ and $\iota(g)(\gamma(t))$ belongs to the cone $C_{\mathrm{SA}}$, that is
\begin{align*}
    \mathbb{1}_{\{g\gamma(t) \in C_{\mathrm{SA}}\}}
    + \mathbb{1}_{\{\iota(g)\gamma(t) \in C_{\mathrm{SA}}\}} = 1.
\end{align*}
Summing over $g \in \mathcal{N}_\partial$, dividing by $2$ and using the fact that $\iota: \mathcal N_\partial\to \mathcal N_\partial$ is a bijection, gives~\eqref{eq:sparre_proof_sum_over_M_ell_half}.  Consequently, the contribution of the elements of $\mathcal N_\partial$ to
$|\mathcal N(\gamma(t))|$ is the same on both sides of the exceptional
time $t=0$. Together with the constancy of the first sum in the above
decomposition, this shows that $|\mathcal N(\gamma(t))|$ has the same
value for all sufficiently small positive and negative $t$. Repeating
this argument at every exceptional time $t_1,\dots,t_L$, and using the
constancy between consecutive exceptional times, we conclude that
$|\mathcal N(x)|=|\mathcal N(y)|$.

\medskip
\textit{Computation of the constant.}
Since $|\mathcal N(x)|$ takes the same value --- we denote it by $\kappa_n$ --- at every regular
vector $x\in\R^n$, we may compute it on a single well-chosen regular vector $x$. We use the tuple
\begin{align*}
    \lambda = (1,2,4,\dots,2^{n-1}),
\end{align*}
whose key property is \emph{domination}: each entry exceeds the sum of all
smaller ones, because $2^k>\sum_{j=0}^{k-1}2^j=2^k-1$ for all $k\geq 1$.

An element $g\in \mathfrak{B}_n$ arranges the entries of $\lambda$ in some order and
attaches a sign to each, producing a tuple
\begin{align*}
    g \lambda = (\varepsilon_1 2^{\sigma(1)-1}, \dots, \varepsilon_n 2^{\sigma(n)-1}),
    \qquad \varepsilon_i\in\{\pm1\},
\end{align*}
where $(\sigma(1),\dots,\sigma(n))$ is a permutation of $(1,\dots,n)$.
Domination has two immediate consequences. First, no signed partial sum
$s_{m}(g\lambda) = \sum_{i\le m}\varepsilon_i 2^{\sigma(i)-1}$ vanishes, since the largest power of $2$ occurring has magnitude greater than the sum of all smaller powers.   So $\lambda$ is regular.
Second, we obtain the following sign rule:
\begin{center}
\emph{The sign of $s_{m}(g\lambda)$ equals the sign attached to the largest
entry among $2^{\sigma(1)-1},\dots,2^{\sigma(m)-1}$,}
\end{center}
since that entry outweighs all the others combined. Thus $\kappa_n$ counts
the signed permutations $g\lambda$ in which, for every $m$, the largest of the first
$m$ entries carries the sign $+$. Recall that a position $i$ is a record of $\sigma$ if
$\sigma(i)>\sigma(j)$ for every $j<i$. So $\kappa_n$ counts
the signed permutations $g\lambda$ in which all record positions carry the sign $+$.

Now there are several  ways to complete the argument.  The number of permutations of
$[n]$ with exactly $k$ records is the unsigned Stirling number of the
first kind $\stirling{n}{k}$, characterized by
$
\sum_{k=0}^n \stirling{n}{k}t^k
=t(t+1)\cdots(t+n-1)$.
If each record has to carry the sign $+$, there are $2^{n-k}$ ways of attaching signs to non-records. So
$$
\kappa_n = \sum_{k=1}^n 2^{n-k} \stirling{n}{k} = 2^n \prod_{j=0}^{n-1} \left(\frac 12 + j\right) = (2n-1)!!,
$$
where we used the generating function for Stirling numbers with $t=1/2$.

Alternatively, we can prove the recurrence relation $\kappa_n = (2n-1)\kappa_{n-1}$ for all $n\geq 2$ by arguing as follows. Delete the entry $1$ (together with its sign) from the signed permutation $g\lambda$.
Since $1$ is the largest among the first $m$ entries only when it stands
first and $m=1$, the deletion affects the sign rule in no other partial
sum: what remains is a signed arrangement of $(2,4,\dots,2^{n-1})$ --- again
a dominated tuple, of length $n-1$ --- satisfying the same positivity
condition. After division by $2$, this is precisely an admissible signed
arrangement of $(1,2,\dots,2^{n-2})$ counted by $\kappa_{n-1}$. Conversely, every admissible arrangement of length $n$ is
obtained from an admissible one of length $n-1$ by inserting $\pm1$
somewhere, and we count the ways to do so:
\begin{itemize}
    \item inserted at one of the $n-1$ positions after the first, the entry
    $1$ is never the running maximum, so both signs are allowed:
    $2(n-1)$ ways;
    \item inserted at the first position, it must satisfy
    $s_1=\varepsilon>0$, so only the sign $+$ is allowed: $1$ way.
\end{itemize}
Hence
\begin{align*}
    \kappa_n = \bigl(2(n-1)+1\bigr)\,\kappa_{n-1}
             = (2n-1)\,\kappa_{n-1}, \qquad n\geq 2.
\end{align*}
Together with the initial condition $\kappa_1 = 1$ this gives $\kappa_n = (2n-1)!! = \frac{(2n)!}{2^n\,n!}$.
\end{proof}

The probabilistic version of Sparre Andersen's formula can now be derived in a standard way.
\begin{theorem}[Sparre Andersen's formula: Probabilistic version]
Let $n\geq1$. Let $\xi_1,\dots,\xi_n$ be real-valued,
possibly dependent random variables with partial sums $S_k:= \xi_1+\dots+\xi_k$, $1\leq k\leq n$, and suppose that the following conditions are satisfied:
\begin{itemize}
\item [(i)]  \emph{Symmetric exchangeability:} For every permutation $\sigma$ of the set $\{1,\dots,n\}$ and every choice of signs $\eps_1,\dots,\eps_n\in \{-1,+1\}$, the following equality in distribution holds:
    $$
    (\xi_1,\dots,\xi_n) \stackrel{d}{=}(\eps_1 \xi_{\sigma(1)}, \dots, \eps_n \xi_{\sigma(n)}).
    $$
\item[(ii)] \emph{No partial-sum atoms at $0$:}
For every $1\leq i \leq n$, we have $\pr[S_i = 0] = 0$.
\end{itemize}
Then
    \begin{align*}
        \pr(S_1 > 0, S_2 > 0, \dots, S_n > 0) = \frac{1}{2^{2n}}
        \binom{2n}{n}.
    \end{align*}
\end{theorem}
\begin{proof}
Let $\mu$ be the joint law of $(\xi_1,\dots,\xi_n)$ on $\R^n$. Then
$\mu$ is $\mathfrak{B}_n$-invariant. Hence, for every
$g\in\mathfrak{B}_n$ and $m\in\{1,\dots,n\}$,
\[
\mu(L_{g,m})
=
\pr\bigl(s_m(g(\xi_1,\dots,\xi_n))=0\bigr)
=
\pr(S_m=0)
=
0.
\]
Since $\mathcal L$ is finite, it follows that
$\mu(\bigcup\mathcal L)=0$. Writing
$\kappa_n=\frac{(2n)!}{2^n n!}$ for the common value of
$|\mathcal N(x)|$ at regular points $x\in\R^n$, we have
\begin{align*}
    \pr(S_1 \ge 0, \dots, S_n \ge 0)
    = \int\limits_{\R^n} \mathbb{1}_{\{x \in C_{\mathrm{SA}}\}}\,\mu(dx)
    = \frac{1}{2^n n!}
       \int\limits_{\R^n} |\mathcal N(x)|\,\mu(dx)
    = \frac{\kappa_n}{2^n n!} = \frac{1}{2^{2n}}\binom{2n}{n}.
\end{align*}
Since each event $\{S_k=0\}$ has probability zero, the same value is
obtained with strict inequalities.
\end{proof}

\section{Wendel's theorem}\label{sec:wendel}

In this section, we give a second illustration of the wall-crossing  method, proving Wendel's formula stated in Theorem~\ref{thm:wendel}.  We begin by introducing notation needed to state its combinatorial version.

Let $n, d \ge 1$ be integers with $n \geq d+1$. Let $x_1,\dots, x_n$ be vectors in $\R^d$ and write $X = (x_1 \,|\, \dots \,|\, x_n) \in \R^{d\times n}$ for the matrix whose columns are these vectors. Consider the polytope
$$
H(X) := \Conv\big(x_1, \dots,x_n\big).
$$
The relevant group here is $(\mathbb Z / 2\mathbb Z)^n$, acting on $\R^{d\times n}$ by flipping the signs of the columns. An element of this group, represented by  $\eps = (\eps_1,\dots, \eps_n) \in \{\pm1\}^n$, maps $X\in \R^{d\times n}$ to
$$
\eps X := (\eps_1 x_1\,|\,  \dots \, | \,  \eps_n x_n)\in \R^{d\times n}.
$$
To every $X\in \R^{d\times n}$ we associate $2^n$ polytopes $H(\eps X) = \Conv(\eps_1 x_1, \dots, \eps_n x_n)$,  $\eps\in \{\pm 1\}^n$, and let
\begin{align*}
    \mathcal M(X) := \big\{ \eps \in \{\pm1\}^n : 0 \notin H(\eps X) \big\}.
\end{align*}
Wendel's formula states that the size of the set $\mathcal M(X)$ is the same for all $X$ outside the following exceptional set. For a $d$-element set $I = \{i_1 < \dots < i_d\} \subseteq [n]$ define the \emph{wall}
$$
    D_{I} := \{ X \in \R^{d\times n} \mid \det\big(x_{i_1}\,|\, \dots\,|\, x_{i_d}\big) = 0 \}.
$$
The exceptional set $\mathcal D$ is then defined as the union of all walls,
$$
\mathcal D := \bigcup_{|I| = d} D_{I}.
$$
We call a configuration $X\in \R^{d\times n}$ \emph{regular} if $X\notin \mathcal D$. Equivalently, $X$ is regular if any $d$ of its columns are linearly independent, that is, if the columns of $X$ are in general linear position.

\begin{theorem}[Wendel's formula: Combinatorial version] \label{thm:wendel_comb}
Let $n, d \ge 1$ be integers with $n \geq d+1$.
For every regular $X \in \R^{d\times n}$, the cardinality of the set $\mathcal M(X)$ is given by
\begin{align*}
    |\mathcal M(X)| = 2 \sum_{j=0}^{d-1} \binom{n-1}{j}.
\end{align*}
In particular, this cardinality is the same for all regular $X$.
\end{theorem}
\begin{remark}\label{rem:wendel_for_n_leq_d}
The conclusion remains valid for $1\leq n\leq d$ if
$x_1,\dots,x_n$ are in general linear position, which in this case
means that they are linearly independent. Indeed,
$0\notin H(\eps X)$ for every $\eps\in\{\pm1\}^n$, and hence
$
|\mathcal M(X)|=2^n
=2\sum_{j=0}^{d-1}\binom{n-1}{j}
$,
where $\binom{n-1}{j}=0$ for $j>n-1$.
\end{remark}

\begin{proof}[Proof of Theorem~\ref{thm:wendel_comb}]
In the first part of the proof we show that $|\mathcal M(X)| = |\mathcal M(Y)|$ for any
two regular $X, Y \in \R^{d\times n}$. In the second part we identify the value of the constant by evaluating $|\mathcal M(X)|$ at a suitably chosen regular configuration $X$.

\medskip
\emph{Constancy along a path.}
Let $X,Y\in \R^{d\times n}$ be any two regular configurations. By Proposition~\ref{prop:pathexist_for_wendel}, whose proof will be given in Appendix~\ref{sec:paths_existence}, there is a piecewise linear path
$\gamma : [-1,1] \to \R^{d\times n}$ such that
\begin{itemize}
\item $\gamma$ connects $X$ to $Y$, that is $\gamma(-1) =X$ and $\gamma(1) = Y$;
\item $\gamma(t) \in \mathcal D$ only at finitely many exceptional times $t_1,\dots, t_L$ with  $-1 < t_1 < \dots < t_L< 1$;
\item  $\gamma$ avoids double degeneracies, i.e.\ for every $\ell=1,\dots, L$, the point $\gamma(t_\ell)$ belongs to \emph{exactly one} wall.
\end{itemize}

We denote the columns of $\gamma(t)$ by $x_1(t), \dots, x_n(t)$ and study the behavior of the function $t \mapsto |\mathcal M(\gamma(t))|$ when $t$ changes from $-1$ to $1$.
We first record a continuity property:
\begin{center}
\emph{if $0 \notin \partial H(\eps Y)$ for some $Y \in \R^{d\times n}$ and $\eps \in \{\pm 1\}^n$,
\\
then the indicator function $Z\mapsto \mathbb{1}_{\{0 \notin H(\eps Z)\}}$ is constant in a neighbourhood of $Y$.}
\end{center}

This is a special case of Lemma~\ref{lem:stability} in the Appendix,
applied to the tuple $(\eps_1 y_1, \dots, \eps_n y_n)$, where
$y_1, \dots, y_n$ denote the columns of $Y$.

Now let $t$ vary. If $\gamma(t)$ is regular, then $0 \notin \partial H(\eps\gamma(t))$ for every $\eps\in \{\pm 1\}^n$: this follows from Lemma~\ref{lem:separation}(iii).  Together with the continuity property this shows that $|\mathcal M(\gamma(t))|$ is constant on each of the intervals $[-1, t_1), (t_1,t_2), \dots, (t_L, 1]$. (Here, we assume that $L\geq 1$;  for $L=0$ the function is constant on $[-1,1]$ and the conclusion follows immediately.)  It remains to verify that $|\mathcal M(\gamma(t))|$ does not change when $t$ crosses an exceptional time $t_\ell$ (although the value exactly at $t_\ell$ may be different).

Fix an exceptional time $t_\ell$ with $\ell \in \{1,\dots, L\}$. After a reparametrization of the path we may assume that $t_\ell = 0$.  Our aim is to show that the function $t\mapsto |\mathcal M (\gamma(t))|$ stays constant for sufficiently small $t \neq 0$.

Let $I\subseteq [n]$ be the unique $d$-set with $\gamma(0) \in D_I$.
By the continuity property, the function $t\mapsto \mathbb{1}_{\{0\notin H(\eps\gamma(t))\}}$ stays constant for sufficiently small $t$ provided $\eps$ is such that  $0 \notin \partial H(\eps\gamma(0))$.  So, let
$$
\mathcal M_\partial := \{\eps \in \{\pm 1\}^n : 0 \in \partial H(\eps\gamma(0))\}.
$$
Our aim is to show that for  $t$ sufficiently close to but not equal to $0$,
\begin{equation}\label{eq:wendel_proof_sum_half}
\sum_{\eps \in \mathcal M_\partial} \mathbb{1}_{\{0 \notin H(\eps\gamma(t))\}} = \frac 12 |\mathcal M_\partial|.
\end{equation}

Fix $\eps \in \mathcal M_\partial$ and consider the tuple
$(\eps_1 x_1(0), \dots, \eps_n x_n(0))$. Sign flips do not affect
linear dependence, so the degenerate $d$-sets of this tuple are those of
$\gamma(0)$; since $\gamma(0)$ lies on exactly one wall, the tuple has
exactly one degenerate $d$-set, namely $I$ --- note that it is the same
set for all $\eps \in \mathcal M_\partial$. Since moreover
$0 \in \partial H(\eps\gamma(0))$, Lemma~\ref{lem:separation}(iv)
applies to this tuple: the points $\eps_i x_i(0)$, $i \in I$, are
affinely independent, $0 \in \relint F_\eps$ for the $(d-1)$-simplex
$F_\eps := \Conv\big(\eps_i x_i(0) : i \in I\big)$, and there is a
unique unit vector $\nu_\eps$ such that
$$
\langle \nu_\eps, \eps_i x_i(0)\rangle = 0 \ \text{ for } i \in I
\qquad\text{and}\qquad
\langle \nu_\eps, \eps_m x_m(0)\rangle > 0 \ \text{ for } m \notin I.
$$

Now comes the key step of the proof. Define the flip $\iota : \{\pm 1\}^n \to \{\pm 1\}^n$ by
$$
(\iota(\eps))_i :=
\begin{cases}
-\eps_i, & i \in I,\\
\phantom{-}\eps_i, & i \notin I.
\end{cases}
$$
Clearly, $\iota \circ \iota = \mathrm{id}$, i.e.\ $\iota$ is an involution of $\{\pm 1\}^n$. Passing from $\eps$ to $\iota(\eps)$ replaces the points $\eps_i x_i(t)$, $i \in I$, by their negatives and leaves the points $\eps_i x_i(t)$, $i \notin I$, unchanged, for \emph{all} $t$. At time $0$ this has the following consequences.

First,
$F_{\iota(\eps)} = -F_\eps$, so $0 \in \relint F_{\iota(\eps)}$
and, in particular, $0 \in H(\iota(\eps)\gamma(0))$.
Second,
$\langle \nu_\eps, -\eps_i x_i(0)\rangle = 0$ for $i \in I$, while
$\langle \nu_\eps, \eps_i x_i(0)\rangle > 0$ for $i \notin I$;
hence, by Lemma~\ref{lem:separation}(ii) applied with $v = \nu_\eps$,
$0 \notin \Int H(\iota(\eps)\gamma(0))$, and therefore $0 \in \partial H(\iota(\eps)\gamma(0))$. Together this gives
$\iota(\eps) \in \mathcal M_\partial$ and, by the uniqueness in
Lemma~\ref{lem:separation}(iv), $\nu_{\iota(\eps)} = \nu_\eps$.
In particular, $\iota$ restricts to an involution, and hence a
bijection, of $\mathcal M_\partial$.

For small $|t|$ the points $\eps_i x_i(t)$, $i \in I$, remain affinely independent; choose a unit normal $\nu(t)$ of their affine hull continuously in $t$ with $\nu(0) = \nu_\eps$, fix any $i\in I$ and set
\begin{align*}
    \psi_\eps(t) := \langle \nu(t),\, \eps_i x_i(t) \rangle.
\end{align*}
This value is independent of the choice of $i\in I$. Note that $\psi_\eps(t) \neq 0$ for small $t \neq 0$: otherwise the origin would lie in the affine hull of the points $\eps_i x_{i}(t)$, $i \in I$, making these vectors linearly dependent and contradicting the regularity of $\gamma(t)$. We claim that for all small enough $|t|$,
\begin{align}\label{eq:wendel_membership}
    0 \in H(\eps\gamma(t)) \iff \psi_\eps(t) \le 0.
\end{align}

Indeed, for $m \notin I$ the inequality $\langle \nu(t), \eps_m x_m(t)\rangle > \psi_\eps(t)$ holds strictly at $t = 0$ and hence nearby, so $H(\eps\gamma(t))$ lies in the half-space $\{z \in \R^d \mid \langle \nu(t), z \rangle \ge \psi_\eps(t)\}$; this gives ``$\Rightarrow$''. Conversely, fix some $m_0 \notin I$. The $d+1$ points $\eps_i x_i(t)$, $i \in I$, and $\eps_{m_0} x_{m_0}(t)$ are affinely independent for small $|t|$, so the barycentric coordinates of the origin with respect to them are well defined and continuous in $t$; at $t = 0$ they are positive on $I$ and zero at $m_0$, because $0 \in \relint F_\eps$. Taking inner products with $\nu(t)$, the $m_0$-coordinate $\lambda(t)$ satisfies $\lambda(t)\big(\langle \nu(t), \eps_{m_0} x_{m_0}(t)\rangle - \psi_\eps(t)\big) = -\psi_\eps(t)$,  where the bracketed
factor is positive. Hence, if $\psi_\eps(t) \le 0$, all the barycentric coordinates are nonnegative and $0 \in H(\eps\gamma(t))$, which gives ``$\Leftarrow$''.

Finally, we run the same construction for $\iota(\eps)$. The points of the corresponding tuple with indices in $I$ are $-\eps_i x_{i}(t)$; their affine hull is $-\,\operatorname{aff}\big(\eps_i x_{i}(t) : i \in I\big)$ and carries the same unit normals, and at $t = 0$ the normal provided by Lemma~\ref{lem:separation}(iv) is $\nu_{\iota(\eps)} = \nu_\eps = \nu(0)$. Hence the same continuous choice $\nu(t)$ serves for $\iota(\eps)$ as well, and
\begin{align*}
    \psi_{\iota(\eps)}(t) = \langle \nu(t),\, -\eps_i x_{i}(t) \rangle = -\,\psi_\eps(t)
\end{align*}
for all small $|t|$. By~\eqref{eq:wendel_membership}, applied to $\eps$ and to $\iota(\eps)$, and since $\psi_\eps(t) \neq 0$ for small $t \neq 0$, exactly one of the polytopes $H(\eps\gamma(t))$ and $H(\iota(\eps)\gamma(t))$ contains the origin, that is
\begin{align*}
    \mathbb{1}_{\{0 \notin H(\eps\gamma(t))\}}
    + \mathbb{1}_{\{0 \notin H(\iota(\eps)\gamma(t))\}} = 1
\end{align*}
for all sufficiently small $t \neq 0$. Summing this equality over $\eps \in \mathcal M_\partial$ and using that $\iota : \mathcal M_\partial \to \mathcal M_\partial$ is a bijection gives
\begin{equation*}
2\sum_{\eps \in \mathcal M_\partial} \mathbb{1}_{\{0 \notin H(\eps\gamma(t))\}} = |\mathcal M_\partial|.
\end{equation*}
Dividing by $2$ gives~\eqref{eq:wendel_proof_sum_half}. Since the right-hand side does not depend on $t$, this shows that the function $t\mapsto |\mathcal M (\gamma(t))|$ stays constant for sufficiently small $t \neq 0$. Applying this claim to every crossing time $t_\ell$ proves that $|\mathcal M(X)| = |\mathcal M(Y)|$.

\medskip
\emph{Computation of the constant.} Since $|\mathcal M(X)|$ takes the same value at every regular configuration $X$, we may compute it on a single well-chosen one. Take real numbers $\tau_1 < \dots < \tau_n$ and let
$$
x_{i} := (1,\, \tau_i,\, \tau_i^2,\, \dots,\, \tau_i^{d-1}),
\qquad 1 \le i \le n,
$$
be points on the moment curve. Any $d$ columns of the resulting matrix $X$ form a Vandermonde matrix with pairwise distinct nodes, so $X$ is regular. By Lemma~\ref{lem:separation}(i),
$0 \notin H(\eps X)$ if and only if there is $v \in \R^d$ with
$\eps_i \langle v, x_i \rangle > 0$ for all $i$, i.e.\ if and only if there is a
polynomial $P$ with $\deg P \le d-1$ and $\sgn P(\tau_i) = \eps_i$ for all $i$.
Such $P$ exists if and only if the sequence $(\eps_1, \dots, \eps_n)$ has at most
$d-1$ sign changes: a nonzero polynomial of degree $\le d-1$ has at most $d-1$ sign
changes along $\tau_1 < \dots < \tau_n$; conversely, a pattern with $k \le d-1$
changes is realized by $\pm\prod_{\ell=1}^{k}(\tau - s_\ell)$ with one root $s_\ell$ placed
in each gap carrying a change. Counting the patterns by the initial sign and by the
set of gaps carrying a change gives
\begin{align*}
    |\mathcal M(X)| = 2 \sum_{j=0}^{d-1} \binom{n-1}{j}.
\end{align*}
This completes the proof.
\end{proof}

\begin{theorem}[Wendel's formula: Probabilistic version] \label{thm:wendel_prob}
Let $\xi_1, \dots, \xi_n$ be random vectors in $\R^d$ whose joint distribution is invariant under sign flips, that is,
$$
(\xi_1, \dots, \xi_n) \stackrel{d}{=} (\eps_1 \xi_1, \dots, \eps_n \xi_n)
\qquad \text{for every } \eps \in \{\pm 1\}^n.
$$
Suppose additionally that the configuration $X := (\xi_1 \,|\, \dots \,|\, \xi_n)$ is regular almost surely, i.e.\ $\pr(X \in \mathcal D) = 0$. Then
\begin{align*}
    \pr\big( 0 \notin \Conv(\xi_1, \dots, \xi_n) \big)
    = \frac{1}{2^{\,n-1}} \sum_{j=0}^{d-1} \binom{n-1}{j}.
\end{align*}
Note that neither independence nor exchangeability of $\xi_1, \dots, \xi_n$ is assumed.
\end{theorem}

\begin{proof}
Let $\mu$ be the law of $X$ on $\R^{d\times n}$. By assumption, $\mu$ is invariant under every sign flip and $\mu(\mathcal D) = 0$. The sign flip invariance implies that $\pr\big(0 \notin \Conv(\xi_1, \dots, \xi_n)\big) = \pr\big(0 \notin H(\eps X)\big)$ for every $\eps \in \{\pm 1\}^n$. Averaging over the group and applying Theorem~\ref{thm:wendel_comb} gives
\begin{align*}
    \pr\big(0 \notin \Conv(\xi_1, \dots, \xi_n)\big)
    &= \frac{1}{2^n} \sum_{\eps \in \{\pm1\}^n} \pr\big(0 \notin H(\eps X)\big)
    \\
    &= \frac{1}{2^n} \int\limits_{\R^{d\times n}} |\mathcal M(Y)|\, \mu(dY)
    = \frac{1}{2^{\,n-1}} \sum_{j=0}^{d-1} \binom{n-1}{j}. \qedhere
\end{align*}
\end{proof}

\medskip
\emph{Characterization of the exceptional set}.
The exceptional set entered the above proof mainly through one property:
for $X\notin \mathcal D$ the origin lies on the boundary of no $H(\eps X)$. Excluding all of
$\mathcal D$ may look wasteful, since the vectors involved in a linear dependence that appears in the definition of $\mathcal D$
may lie deep inside $H(\eps X)$, and its  boundary need not contain the origin. The next proposition shows
that nothing is gained by excluding less: the freedom in the choice of $\eps$
turns any linear dependence into a boundary degeneracy.

\begin{proposition}[Exceptional set in Wendel's theorem]
\label{prop:wendel_exceptional_set}
For every $X\in\R^{d\times n}$,
\[
    X\in\mathcal D
    \quad\Longleftrightarrow\quad
    0\in\partial H(\eps X)
    \text{ for some }\eps\in\{\pm1\}^n.
\]
\end{proposition}

\begin{proof}
Suppose first that $X\in\mathcal D$. Then there is a $d$-element set
$I\subseteq[n]$ and a nonzero vector $(a_i)_{i\in I}$ such that $\sum_{i\in I}a_i x_i=0$.

Let $J:=\{i\in I:a_i\neq0\}$ and put $\eps_i:=\sgn(a_i)$ for $i\in J$.
After normalization, the above relation shows that
\[
    0\in\Conv(\eps_i x_i:i\in J).
\]
Since the vectors $(x_i)_{i\in I}$ are linearly dependent, there is
$v\neq0$ orthogonal to all of them. Complete $\eps$ so that
$\langle v,\eps_m x_m\rangle\geq0$ for every $m\notin I$, choosing the
 signs on $I\setminus J$ arbitrarily. Then $0\in\Conv(\eps_i x_i:i\in J) \subseteq H(\eps X)$ and  all vertices of $H(\eps X)$ lie in
the closed half-space $\{z:\langle v,z\rangle\geq0\}$.
Hence $v^{\perp}$ is a supporting hyperplane of $H(\eps X)$ and $v^{\perp}$ contains $0$. This implies $0\in \partial H(\eps X)$.

Conversely, suppose that $X\notin\mathcal D$. Then the columns of
$\eps X$ are in general linear position for every $\eps$. Hence, by
Lemma~\ref{lem:separation}(iii), the origin is either outside
$H(\eps X)$ or belongs to its interior. Thus
$0\notin\partial H(\eps X)$ for every $\eps$, which proves the
contrapositive.
\end{proof}

\section{Absorption by convex hulls of random walks}\label{sec:convex_hulls_random_walks}

In this section, we prove the combinatorial result behind Theorem~\ref{thm:1}.
Let $d\geq 1$ and $n \ge d+1$. Let $X = (x_1\,|\,\dots\,|\,x_n) \in \R^{d\times n}$ be the matrix with columns $x_1, \dots, x_n \in \R^d$, and let
$$
s_k(X) := x_1 + \dots + x_k, \qquad  1 \le k \le n,\qquad s_0(X):=0,
$$
be  the partial sums of
its columns. The convex hull of these partial sums is a polytope denoted by
$$
H(X) := \Conv\big(s_1(X), \dots, s_n(X)\big)\subseteq \R^d.
$$
The underlying symmetry group is  $\mathfrak{B}_n$, the group of signed permutations of $n$ elements. In this section, this group acts on   $\R^{d\times n}$ by permuting
the columns and multiplying an arbitrary subset of them by $-1$. More precisely, we parametrize every element $g\in \mathfrak{B}_n$ by a pair $(\sigma, \eps)$, where $\sigma$ is a permutation  of $[n]$ and   $\eps\in \{\pm1\}^n$ is a vector of signs. Then $g: \R^{d\times n} \to \R^{d\times n}$ is a linear transformation defined by
$$
gX := (\eps_1 x_{\sigma(1)} \, |\, \dots \, |\, \eps_n x_{\sigma(n)}).
$$
For every $X\in \R^{d\times n}$ we consider $2^n n!$ polytopes of the form $H(gX)$, $g\in \mathfrak{B}_n$,  and  define
\begin{align*}
    \mathcal M(X) := \{ g \in \mathfrak{B}_n \mid 0 \notin H(gX) \}.
\end{align*}

The main combinatorial result of this section states that the size of the set $\mathcal M(X)$ is
the same for all $X$ outside an exceptional set $\mathcal D \subseteq \R^{d\times n}$ defined as follows. For
$g \in \mathfrak{B}_n$ and a $d$-element set $I \subseteq [n]$ define the
\emph{wall}
\begin{align} \label{eq:walldef}
    D_{g,I} := \{ X \in \R^{d\times n} \mid
    \text{the vectors } s_i(gX),\ i \in I, \text{ are linearly dependent} \},
\end{align}
and let the exceptional set $\mathcal D$ be the union of all such walls,
$$\mathcal D := \bigcup_{g \in \mathfrak{B}_n} \bigcup_{|I|=d} D_{g,I}.
$$
We call $X \in \R^{d\times n}$ \emph{regular} if $X \notin \mathcal D$.
Equivalently, $X$ is regular if for every $g \in \mathfrak{B}_n$ the partial
sums $s_1(gX), \dots, s_n(gX)$ are in general linear position. For $d = 1$
this recovers the notion of regularity used in Section~\ref{sec:Sparre-Andersen}.

\begin{theorem}[Absorption for convex hulls of walks: Combinatorial version]\label{thm:convex_hull_random_walk_comb}
Let $d\geq 1$ and $n \ge d+1$.
Let $x_1,\dots, x_{n}\in \R^d$ be vectors such that   $X= (x_1\, |\, \dots\,| \, x_n)$  is regular. Then,
$$
|\mathcal M(X)| = 2\sum_{q\geq 0}
\stirlingb{n}{d-1-2q},
$$
where $\stirlingb{n}{k}$ are the $B$-Stirling numbers defined in Theorem~\ref{thm:1}.
In particular, $|\mathcal M(X)|$ is the same for all $X\in \R^{d\times n}\setminus \mathcal D$.
\end{theorem}

We now turn to the proof of Theorem~\ref{thm:convex_hull_random_walk_comb}, which will occupy most of this section. As in the previous section, the proof consists of two main parts: proving the constancy outside the exceptional set and identifying the value of the constant.

Let $X\in \R^{d\times n}$ and $Y\in \R^{d\times n}$ be regular.
We claim that there is a piecewise linear path $(\gamma(t))_{t\in [-1,1]}$ such that $\gamma(-1) = X, \gamma(1) = Y$ and
\begin{itemize}
    \item[(i)] $\gamma(t)$ is regular for all but finitely many $t$; denote these exceptional times by $t_1,\dots, t_L$ with  $-1 < t_1 < \dots < t_L < 1$;
    \item[(ii)] for every exceptional time $t_\ell$ and every $g \in \mathfrak{B}_n$, there is at most one $d$-element set $I\subseteq [n]$ with $\gamma(t_\ell) \in D_{g,I}$.
\end{itemize}
The existence of the path will be shown in Proposition~\ref{prop:pathexists_general} in  Appendix~\ref{sec:paths_existence}. To apply it, note that the matrix  $(s_1(gX)\,|\, \dots \,|\, s_n(gX))$ can be represented as $XB_g$ for a suitable invertible matrix $B_g \in \operatorname{GL}_n(\R)$.

\medskip
\emph{Constancy along a path.}
Our aim is to prove that $|\mathcal M(X)| = |\mathcal M(Y)|$.
By Lemma~\ref{lem:separation}~(iii) in
Appendix~\ref{convex_geometry}, for $t\notin \{t_1,\ldots, t_L\}$,
 we have $0\notin \partial H(g\gamma(t))$. Hence, by Lemma~\ref{lem:stability} in Appendix~\ref{convex_geometry}, for every $g\in \mathfrak{B}_n$,  the function
$t \mapsto \mathbb{1}_{\{0 \in H(g\gamma(t))\}}$ is constant on each connected component of $[-1,1]\setminus \{t_1,\dots, t_L\}$.
It follows that the value $|\mathcal M(\gamma(t))|$ is also constant on each connected component.  Thus, it suffices to show that for each $\ell$ the value $|\mathcal M(\gamma(t))|$ does not change as $t$ crosses $t_\ell$ (although the value exactly at $t_\ell$ may be different). It suffices to consider a path with only one exceptional time $t=0$. Our task reduces to proving the following result.

\begin{proposition}[Wall-crossing invariance]\label{prop:crossing}
Let $\gamma : [-1,1] \to \R^{d\times n}$ be a piecewise linear path such that $\gamma(0)\in \mathcal D$ and $\gamma(t)$ is regular for all $t \neq 0$. Suppose also that  for every $g \in \mathfrak{B}_n$ there is at most one $d$-element set $I\subseteq [n]$ with $\gamma(0) \in D_{g,I}$. Then the function
$$
t\mapsto |\mathcal M(\gamma(t))|
$$ stays constant for all $t\neq 0$ with sufficiently small $|t|$.
\end{proposition}

\begin{proof}

Put $\mathcal M_\partial := \{ g \in \mathfrak{B}_n \mid 0 \in \partial H(g\gamma(0)) \}$.
By Lemma~\ref{lem:stability}, for $g \notin \mathcal M_\partial$ the indicator
$\mathbb 1_{\{0 \in H(g\gamma(t))\}}$ is constant for small $|t|$. (The origin stays inside a convex hull after a small perturbation if it was strictly inside at time $0$. Similarly, the origin stays outside if it was outside at time $0$.) Therefore, it suffices to verify that $\big|\mathcal M (\gamma(t)) \cap \mathcal M_\partial\big|$ stays constant for all sufficiently small $t\neq 0$.  We shall show a stronger statement, namely
\begin{align} \label{Mpartialnum}
    \sum_{g \in \mathcal M_\partial} \mathbb 1_{\{0 \notin H(g\gamma(t))\}}
=
    \frac12 |\mathcal M_\partial|
\end{align}
for all sufficiently small $t\neq 0$. This will be done by constructing an explicit bijection $\iota: \mathcal M_\partial \to \mathcal M_\partial$ such that
\begin{align} \label{eq:twoindicators_0}
    \mathbb 1_{\{0 \notin H(g\gamma(t))\}}
    + \mathbb 1_{\{0 \notin H(\iota(g)\gamma(t))\}} = 1
\end{align}
for all $g\in \mathcal M_\partial$ and all sufficiently small $t\neq 0$. In words, $0$ belongs to exactly one convex hull, $H(g\gamma(t))$ or $H(\iota(g)\gamma(t))$, for small $t\neq 0$.   Summing~\eqref{eq:twoindicators_0} over all  $g\in \mathcal M_\partial$ and dividing by $2$  gives~\eqref{Mpartialnum}.

Fix some $g \in \mathcal M_\partial$ and consider the tuple
$\big(s_1(g\gamma(0)), \dots, s_n(g\gamma(0))\big)$.
By definition of $\mathcal M_\partial$, one has $0 \in \partial H(g\gamma(0))$. By Lemma~\ref{lem:separation}(iii), there is a $d$-element set $I\subseteq [n]$ such that the vectors $s_i(g\gamma(0))$, $i\in I$, are linearly dependent (in other words, $\gamma(0)\in D_{g,I}$).
By the assumption of the proposition, for fixed $g$, such a set $I$ is unique. Denote
it by $I(g) = \{i_1 < \dots < i_d\}$. By
Lemma~\ref{lem:separation}(iv), the points $s_i(g\gamma(0))$,
$i \in I(g)$, are affinely independent, $0 \in \relint F$ for the
$(d-1)$-simplex $F := \Conv\big(s_i(g\gamma(0)) : i \in I(g)\big)$,
and there is a unique unit vector $\nu$ such that
\begin{align}\label{eq:normal_nu_scalar_prod_walks}
    \langle \nu, s_i(g\gamma(0)) \rangle = 0 \ \text{ for } i \in I(g)
    \qquad \text{and} \qquad
    \langle \nu, s_m(g\gamma(0)) \rangle > 0 \ \text{ for } m \notin I(g).
\end{align}

Now comes the key step of the proof.  For $I = \{i_1 < \dots < i_d\} \subseteq [n]$,  define a linear
map $\rho_I : \R^{d\times n} \to \R^{d\times n}$ as follows: for all $X=(x_1 \, |\, \dots \,|\, x_n)\in \R^{d\times n}$ we put
\begin{align*}
    (\rho_I X)_j
    =
    \begin{cases}
        -\,x_{i_k + i_{k+1} + 1 - j},
            & i_k < j \le i_{k+1},\ k \in \{0, \dots, d-1\},\\[2pt]
        x_{j}, & j > i_d,
    \end{cases}
\end{align*}
with the convention $i_0 := 0$.
In words: inside each block $(x_{i_k+1},x_{i_k+2}, \dots, x_{i_{k+1}})$ with $k\in \{0,\dots, d-1\}$, the components
are reversed and their signs are flipped. So, this block transforms according to the rule
$$
(x_{i_k+1},x_{i_k+2}, \dots, x_{i_{k+1}})\mapsto  (-x_{i_{k+1}},-x_{i_{k+1}-1},\dots, -x_{i_k+1}).$$  The tail block $(x_{i_d+1}, \dots,  x_{n})$ remains untouched.

As a direct consequence of the definition,
$\rho_I \in \mathfrak{B}_n$ and $\rho_I \circ \rho_I = \mathrm{id}$.
A telescoping computation gives
\begin{align}\label{eq:sums_rho_I_walks}
s_{\ell}(\rho_I X)
=
\begin{cases}
s_{i_k + i_{k+1} - \ell}(X) - s_{i_k}(X) - s_{i_{k+1}}(X), & i_k < \ell \le i_{k+1},
\\
s_{\ell}(X) - 2\, s_{i_d}(X), &\ell > i_d,
\end{cases}
\end{align}
again with the convention $i_0 := 0$.
In particular,
\begin{equation}\label{eq:sums_rho_I_walks_partial_case}
s_{i}(\rho_I X) = -\,s_{i}(X) \qquad \text{ for all } i \in I.
\end{equation}

For $g\in \mathcal M_\partial$ we define $\iota(g):=\rho_{I(g)}g\in \mathfrak{B}_n$.  We now prove that $\iota$ maps $\mathcal M_\partial$ to itself and $\iota \circ \iota = \mathrm{id}$ on $\mathcal M_\partial$.   By~\eqref{eq:sums_rho_I_walks_partial_case},
$$
s_i (\iota(g) \gamma(0)) = - s_i (g \gamma(0)), \qquad i\in I(g).
$$
We now compute the inner products of $\nu$ with  $s_1(\iota(g)\gamma(0)), \dots,  s_n(\iota(g)\gamma(0))$. By~\eqref{eq:sums_rho_I_walks_partial_case} and~\eqref{eq:normal_nu_scalar_prod_walks},
\begin{align*}
    \langle \nu, s_i(\iota(g)\gamma(0)) \rangle
    =
    -\langle \nu, s_i(g\gamma(0)) \rangle
    = 0, \quad \text{ for all }  i \in I(g).
\end{align*}
Next, it follows from~\eqref{eq:sums_rho_I_walks} and~\eqref{eq:normal_nu_scalar_prod_walks} that for all $\ell\in [n]\setminus I(g)$,
\begin{align} \label{eq:nu_positive_scalar_product_iota_walk}
    \langle \nu, s_\ell(\iota(g)\gamma(0)) \rangle >0.
\end{align}
It follows that the simplex
$$
\Conv\big(s_i(\iota(g)\gamma(0)) : i \in I(g)\big)
=
\Conv\big(-s_i(g\gamma(0)) : i \in I(g)\big) =
-F
$$
is a face of $H(\iota(g)\gamma(0))$ with unit normal vector $\nu$. As we already know,  $0 \in \relint F$. Hence, $0\in \relint (-F)$ and, in particular, $0\in \partial H(\iota(g) \gamma(0))$. This shows that $\iota(g) \in \mathcal M_\partial$ and, moreover, $I(\iota(g)) = I(g)$. Since $\rho_I$ is an involution for every $d$-element set $I\subseteq [n]$, we conclude that  $\iota \circ \iota = \mathrm{id}$ on $\mathcal M_\partial$. In particular, $\iota: \mathcal M_\partial \to \mathcal M_\partial$ is a bijection.

We now prove the key property of the pairing $g \leftrightarrow \iota(g)$: for small nonzero $t$, exactly one of the polytopes  $H(g\gamma(t))$ and $H(\iota(g)\gamma(t))$ contains the origin, or, equivalently,
\begin{align} \label{eq:twoindicators}
    \mathbb 1_{\{0 \notin H(g\gamma(t))\}}
    + \mathbb 1_{\{0 \notin H(\iota(g)\gamma(t))\}} = 1.
\end{align}
As already explained, this is the only missing ingredient in the proof of   Proposition~\ref{prop:crossing}.

Fix $g \in \mathcal M_\partial$, write $I := I(g)$ and
$s_{m}(t) := s_{m}(g\gamma(t))$. For small $|t|$ the points $s_{i}(t)$,
$i \in I$, remain affinely independent; choose a unit normal $\nu(t)$ of their
affine hull continuously in $t$ with $\nu(0) = \nu$, take some $i\in I$ and set
\begin{align*}
    \psi_g(t) := \langle \nu(t),\, s_{i}(t) \rangle.
\end{align*}
 This value is independent of the choice of $i\in I$. Note that $\psi_g(t) \neq 0$ for
$t \neq 0$: otherwise the origin would lie in the affine hull of the
$s_{i}(t)$, $i \in I$, making these vectors linearly dependent and
contradicting the regularity of $\gamma(t)$ for $t\neq 0$. We claim that for sufficiently small $|t|$,
\begin{align}\label{eq:membership}
    0 \notin H(g\gamma(t)) \iff \psi_g(t) > 0.
\end{align}

Indeed, for $m \notin I$ the inequality
$\langle \nu(t), s_{m}(t)\rangle > \psi_g(t)$ holds strictly at $t = 0$ and
hence nearby, so the points $s_1(t),\dots, s_n(t)$ are contained in the affine half-space
$\{z \in \R^d \mid \langle \nu(t), z\rangle \ge \psi_g(t)\}$. This implies that their convex hull $H(g\gamma(t))$ is contained in the same half-space and gives
``$\Leftarrow$''. We now prove ``$\Rightarrow$''. Assume that $\psi_g(t) \leq 0$. We need to show that $0\in H(g\gamma(t))$.    Fix $m_0 \notin I$. The $d+1$ points
$s_{i}(t)$, $i \in I$, and $s_{m_0}(t)$ are affinely independent for small
$|t|$, so the barycentric coordinates of the origin with respect to them are
well defined and continuous in $t$; at $t = 0$ they are positive on $I$ and
zero at $m_0$, because $0 \in \relint F$. Taking inner products  with $\nu(t)$, the
$m_0$-coordinate $\lambda(t)$ satisfies
$$
\lambda(t)\big(\langle \nu(t), s_{m_0}(t)\rangle - \psi_g(t)\big)
= -\psi_g(t)
$$
where the bracketed factor is positive.  After shrinking the neighbourhood if necessary, the barycentric
coordinates indexed by $I$ remain strictly positive. If
$\psi_g(t)\leq0$, the displayed identity gives $\lambda(t)\geq0$
(and $\lambda(t)>0$ for $t\neq0$), so all barycentric coordinates
are nonnegative. So $0 \in H(g\gamma(t))$. This gives
``$\Rightarrow$''.

We are now going to apply~\eqref{eq:membership} to $g$ and $\iota(g)$. By \eqref{eq:sums_rho_I_walks_partial_case},
$s_{i}(\iota(g)\gamma(t)) = -\,s_{i}(g\gamma(t))$ for $i \in I(g)$ and all
$t$, so $\nu(t)$ is a unit normal of the affine hull of these points as well,
and at $t = 0$ the remaining points $s_{m}(\iota(g)\gamma(t))$, $m\notin I$, have positive inner products with $\nu(0)$ as shown in~\eqref{eq:nu_positive_scalar_product_iota_walk}. Thus we can run the above construction both for $g$ and $\iota(g)$
 with the same unit normal $\nu(t)$. This  yields
\begin{align*}
    \psi_{\iota(g)}(t)
    = \langle \nu(t),\, -s_{i}(g\gamma(t)) \rangle
    = -\,\psi_g(t).
\end{align*}
Since $\psi_g(t) \neq 0$ for $t \neq 0$, exactly one of
$\psi_g(t)$, $\psi_{\iota(g)}(t)$ is negative. Together with~\eqref{eq:membership}, this  implies~\eqref{eq:twoindicators}. The proof of the  proposition is complete.
\end{proof}

\medskip
\emph{The value of the constant.}
We now identify the constant in Theorem~\ref{thm:convex_hull_random_walk_comb}. This is done by evaluating $|\mathcal M(X)|$ at a well-chosen regular configuration $X$.
We parametrize $g \in \mathfrak{B}_n$ by
a permutation $\sigma$ of $[n]$ and signs $\varepsilon \in \{\pm 1\}^n$,
so that $gX = \big(\varepsilon_1 x_{\sigma(1)}, \dots,
\varepsilon_n x_{\sigma(n)}\big)$. A position $i \in [n]$ is a
\emph{record} of $\sigma$ if $\sigma(j) < \sigma(i)$ for all $j < i$;
we write $\mathrm{Rec}(\sigma)$ for the set of record positions  of $\sigma$, and  $\mathrm{rec}(\sigma) = |\mathrm{Rec}(\sigma)|$ for the number of records.

\begin{proposition}[Reduction to records]\label{prop:lacunary}
    Let $u_1, \dots, u_n$ be unit vectors in $\R^d$ in general linear
    position, and for $R > 1$ set
    \begin{align*}
        \lambda_R := \big(u_1 \, |\, R\,u_2 \, |\,  R^2 u_3\, | \,  \dots \, | \,
        R^{n-1} u_n\big) \in \R^{d\times n}.
    \end{align*}
    Then there is $R_0 = R_0(u_1, \dots, u_n)$ such that for every
    $R \ge R_0$ the configuration $\lambda_R$ is regular and, for every
    $g = (\sigma, \varepsilon) \in \mathfrak{B}_n$,
    \begin{equation}\label{eq:conv_hulls_RW_constant_equivalence_to_records}
        0 \in H(g\lambda_R)
        \iff
        0 \in \Conv\big( \varepsilon_i\, u_{\sigma(i)} \;:\;
           i\in \mathrm{Rec}(\sigma) \big).
    \end{equation}
\end{proposition}

\begin{proof}
Fix a signed permutation $g = (\sigma, \varepsilon)$ in $\mathfrak{B}_n$. Take $1 \le m \le n$ and consider the partial sum
\begin{align*}
    s_m(g\lambda_R)
    = \sum_{i=1}^m \varepsilon_i\, R^{\,\sigma(i)-1}\, u_{\sigma(i)}.
\end{align*}
Let $i(m)$ be the position of the maximum among $\sigma(1), \sigma(2), \dots, \sigma(m)$.  So $i(m) \le m$ and
\begin{align*}
    \sigma(i(m)) \ge \sigma(k),\quad k = 1,2,\dots, m.
\end{align*}

Note that $i(m)$ is the last record position of $\sigma$ among the first $m$ positions.
Every other term carries an exponent smaller by at
least one, so
\begin{align} \label{eq:lacunaryest}
    s_m(g\lambda_R) = R^{\,\sigma(i(m))-1} \big( \varepsilon_{i(m)}
u_{\sigma(i(m))} + \theta_m(R) \big),
    \qquad
    |\theta_m(R)| \le \frac{n}{R}.
\end{align}

The collection of record vectors $\varepsilon_{i}u_{\sigma(i)}$, where $i$ runs through the record positions of $\sigma$, is in general linear position since $u_1, \dots, u_n$ are in general linear position and signs do not affect linear independence.
By
Lemma~\ref{lem:separation}(iii),
\begin{align*}
0 \notin \partial \Conv(\varepsilon_i\, u_{\sigma(i)} \;:\;
            i \in \mathrm{Rec}(\sigma) ).
\end{align*}
Hence the origin is either in the interior or in the complement; we treat the
two cases separately.

\smallskip
\emph{Case 1.}
If the origin does not belong to the convex hull of record vectors, we claim that $0 \notin H(g\lambda_R)$.  Indeed, the separating hyperplane theorem gives a unit vector $v\in \R^d$ with $\langle w, v
\rangle \ge c > 0$ for every record vector $w$. By
\eqref{eq:lacunaryest},
$$\langle s_m(g\lambda_R), v \rangle \ge
R^{\,\sigma(i(m))-1}(c - n/R) > 0$$
for all $m\in [n]$ once $R > n/c$, so
$0 \notin H(g\lambda_R)$.

\smallskip
\emph{Case 2.}
Suppose instead that the origin is in the interior of the record vector convex hull. We prove that $0\in \Int H(g\lambda_{R})$  for sufficiently large $R$.  Suppose to the contrary that
$0 \notin \Int H(g\lambda_{R_j})$ for all $j = 1,2,\ldots$, where  $(R_j)_{j=1}^\infty$ is a sequence such that
$R_j \to \infty$ as $j\to\infty$.  By the separating hyperplane theorem there exist unit vectors
$v_j\in \R^d$ with $\langle s_m(g\lambda_{R_j}), v_j \rangle \ge 0$ for all $m\in [n]$. By compactness, after passing to a subsequence we may assume that $v_j \to v$ for some unit vector $v$. Then by \eqref{eq:lacunaryest}
\begin{align*}
    \frac{1}{R_j^{\sigma(i(m))-1}}
    \langle s_m(g\lambda_{R_j}), v_j \rangle
\to
    \langle \varepsilon_{i(m)}
u_{\sigma(i(m))}, v \rangle.
\end{align*}
Since the left-hand side is nonnegative for all $j=1,2,\ldots$, passing to the limit gives
\begin{align*}
    \langle \varepsilon_{i(m)}
u_{\sigma(i(m))}, v \rangle
\ge 0.
\end{align*}
This contradicts the assumption that the origin is in the interior of the convex hull of record vectors. Hence there is no such sequence $R_j$ and $0 \in \Int H(g\lambda_R)$ for
all sufficiently large $R$.

\smallskip
It remains to verify that $\lambda_R$ is regular for large $R$. Fix $g=(\sigma,\varepsilon)$ and
$I=\{m_1<\dots<m_d\}\subseteq[n]$, and put $m_0:=0$. Successive
column subtractions give
\[
\det\big(s_{m_1}(g\lambda_R),\dots,s_{m_d}(g\lambda_R)\big)
=
\det\left(
 \sum_{i=m_{k-1}+1}^{m_k}
 \varepsilon_iR^{\sigma(i)-1}u_{\sigma(i)}
 :1\leq k\leq d
\right).
\]
In each block $\{m_{k-1}+1,\dots, m_k\}$, let $r_k$ be the unique position at
which $\sigma$ is maximal. The determinant on the right is a
polynomial in $R$ whose unique leading term has coefficient
\[
\left(\prod_{k=1}^d\varepsilon_{r_k}\right)
\det\big(u_{\sigma(r_1)},\dots,u_{\sigma(r_d)}\big)\neq0,
\]
by general linear position of $u_1,\dots, u_n$. It is therefore nonzero for all
sufficiently large $R$. Since there are only finitely many pairs
$(g,I)$, $\lambda_R$ is regular for all sufficiently large $R$.

\smallskip
Combining this observation with Cases~1 and~2 and using the finiteness
of the group $\mathfrak B_n$, we obtain a common threshold $R_0$ such that, for $R\geq R_0$,  the configuration $\lambda_R$ is
regular and the equivalence in~\eqref{eq:conv_hulls_RW_constant_equivalence_to_records} holds for every $g\in\mathfrak B_n$.
\end{proof}

\begin{corollary}
In the setting of Proposition~\ref{prop:lacunary}, for sufficiently large $R$ we have
$$
 |\{g\in \mathfrak{B}_n :   0 \notin H(g\lambda_R)\}| =
 2 \sum_{k=1}^n \left( \stirling{n}{k} 2^{n-k} \sum_{j=0}^{d-1} \binom{k-1}{j}\right),
$$
where $\stirling{n}{k}$ are the unsigned Stirling numbers of the first kind defined by the generating function $\sum_{k=0}^n \stirling{n}{k}\, t^k = t(t+1)\cdots(t+n-1)$.
\end{corollary}
\begin{proof}
Let $\mathrm{rec}(\sigma)= |\mathrm{Rec}(\sigma)|$ be the number of records of a permutation $\sigma$. By Proposition~\ref{prop:lacunary}
\begin{align*}
|\{g\in \mathfrak{B}_n :   0 \notin H(g\lambda_R)\}| = \sum_{\sigma} \sum_{\eps} \mathbb{1}_{\{0\notin \Conv (\eps_i u_{\sigma(i)}: i\in \mathrm{Rec}(\sigma) )\}}.
\end{align*}
Fix some permutation $\sigma$. The number of sign combinations $(\eps_i)_{i\in \mathrm{Rec}(\sigma)}$ for which $0\notin \Conv (\eps_i u_{\sigma(i)}: i\in \mathrm{Rec}(\sigma))$ is given by Wendel's formula,  Theorem~\ref{thm:wendel_comb}, in which the number of points $n$ is replaced by the number of records $\mathrm{rec}(\sigma)$; when $\mathrm{rec}(\sigma)\leq d$ use Remark~\ref{rem:wendel_for_n_leq_d}.  The signs $\eps_i$ with $i\notin \mathrm{Rec}(\sigma)$ are free. So,
$$
|\{g\in \mathfrak{B}_n :   0 \notin H(g\lambda_R)\}| = \sum_{\sigma} \left(2^{n-\mathrm{rec}(\sigma)} \cdot 2 \sum_{j=0}^{d-1} \binom{\mathrm{rec}(\sigma)-1}{j}\right).
$$
Now, for every $k\in [n]$, the number of permutations $\sigma$ having exactly $k$ records is given by the Stirling number $\stirling{n}{k}$. This gives the stated formula.
\end{proof}

\begin{lemma} \label{lem:constants_equivalents}
For integers $n,d\geq 1$ we have
$$
\sum_{\ell=1}^{n}
\stirling{n}{\ell}\,
2^{\,n-\ell}
\sum_{j=0}^{d-1}
\binom{\ell-1}{j}
=
\sum_{q\geq 0}
\stirlingb{n}{d-1-2q},
$$
where the numbers $\stirlingb{n}{k}$ with $n\in \{1,2,\ldots\}$, $k\in \mathbb Z$,  are defined by the generating function
$$
\sum_{k\in \mathbb Z} \stirlingb{n}{k} t^k
=
(t+1)(t+3)\cdots(t+2n-1).
$$
\end{lemma}
\begin{proof}
Using $\binom{\ell-1}{j}
=[t^j](1+t)^{\ell-1}
=[t^j]\frac{(1+t)^\ell}{1+t}$ and interchanging the order of summation,
\begin{align*}
S(n,d) := \sum_{\ell=1}^{n}
\stirling{n}{\ell}\,
2^{\,n-\ell}
\sum_{j=0}^{d-1}
\binom{\ell-1}{j}
&=
\sum_{j=0}^{d-1}
[t^j]
\sum_{\ell=1}^{n}
\stirling{n}{\ell}\,
2^{\,n-\ell}
\frac{(1+t)^\ell}{1+t}.
\end{align*}
Using the generating functions of $\stirling{n}{k}$ and $\stirlingb{n}{k}$, we obtain
\[
\sum_{\ell=0}^n \stirling{n}{\ell}2^{n-\ell}(1+t)^\ell
=(t+1)(t+3)\cdots(t+2n-1)
=\sum_{r=0}^n \stirlingb{n}{r}t^r.
\]
Hence, by $[t^j] \frac{t^r}{1+t} = (-1)^{j-r}$ for $j\geq r$ (the coefficient equals $0$ for $j<r$), we have
\[
S(n,d)
=\sum_{j=0}^{d-1}[t^j]\frac{\sum_{r=0}^n \stirlingb{n}{r}t^r}{1+t}
=\sum_{r=0}^{d-1}\stirlingb{n}{r}\sum_{j=r}^{d-1}(-1)^{j-r}
.
\]
The inner sum equals $1$ if  $d$ and $r$ have different parity, and $0$ otherwise. This gives the stated identity.
\end{proof}

This completes the proof of Theorem~\ref{thm:convex_hull_random_walk_comb}. Its probabilistic version stated in Theorem~\ref{thm:1} can be deduced by averaging over the group $\mathfrak{B}_n$.

\medskip
\emph{Characterization of the exceptional set.}
As in Section~\ref{sec:wendel}, the exceptional set entered the proof mainly
through one property: outside $\mathcal D$ the origin lies on the boundary of no
$H(gX)$. Excluding all of $\mathcal D$ may look wasteful, since the partial sums
involved in a linear dependence may lie deep inside the corresponding polytope.
The next proposition shows that nothing is gained by excluding less: the freedom
in the choice of $g$ turns any linear dependence into a boundary degeneracy,
though possibly for a different signed permutation.

\begin{proposition}[Exceptional set for convex hulls of random walks]\label{prop:walk_exceptional_set}
For every $X\in\R^{d\times n}$,
\[
    X\in\mathcal D
    \quad\Longleftrightarrow\quad
    0\in\partial H(gX)
    \text{ for some }g\in\mathfrak{B}_n.
\]
\end{proposition}

\begin{proof}
Suppose first that $X\in\mathcal D$. Choose $g\in\mathfrak{B}_n$ and
$I=\{i_1<\dots<i_d\}\subseteq[n]$ with $X\in D_{g,I}$, write
$gX=(y_1\,|\,\dots\,|\,y_n)$ and put $i_0:=0$. Split the increments
$y_1,\dots,y_{i_d}$ into consecutive blocks
$B_k:=\{i_{k-1}+1,\dots,i_k\}$ with sums $z_k:=\sum_{j\in B_k}y_j$,
$1\leq k\leq d$. Since $s_{i_k}(gX)=z_1+\dots+z_k$, the vectors
$z_1,\dots,z_d$ are linearly dependent. Hence there is $v\neq0$
orthogonal to all of them, and we may fix a nontrivial relation
$\sum_{k=1}^d c_kz_k=0$. Permuting the blocks and multiplying all
increments of a block by $-1$ amounts to replacing $g$ by another element
of $\mathfrak{B}_n$; doing so and rescaling the relation, we may assume
that $1=c_1\geq\dots\geq c_d\geq0$. Putting $p_k:=z_1+\dots+z_k=s_{i_k}(gX)$,
$p_0:=0$ and $c_{d+1}:=0$, summation by parts gives
\[
    0=\sum_{k=1}^d c_kz_k=\sum_{k=1}^d(c_k-c_{k+1})\,p_k,
\]
where the coefficients $c_k-c_{k+1}$ are nonnegative and sum  to
$c_1=1$. Hence $0\in\Conv(p_1,\dots,p_d)=\Conv(s_{i_1}(gX),\dots,s_{i_d}(gX))$.

It remains to reorder the increments inside the blocks. The goal is to move all partial sums into the half-space $\{z:\langle v,z\rangle\geq0\}$.
Since
$\langle v,z_k\rangle=0$, the numbers $\langle v,y_j\rangle$, $j\in B_k$,
sum to zero, so shifting the block cyclically to start right after a
position where their partial sums attain the minimum, we achieve that all
partial sums inside the block have nonnegative inner product with $v$.
Such a shift does not change $z_k$, and hence does not change
$p_1,\dots,p_d$. Finally, we append the increments $y_j$ with $j>i_d$,
choosing their signs so that their inner products with $v$ are
nonnegative. The resulting signed permutation $\widetilde g$ has the
property that every partial sum of $\widetilde gX$ is of the form
$p_k+w$ with $0\leq k\leq d$ and $\langle v,w\rangle\geq0$; since
$\langle v,p_k\rangle=0$, all vertices of $H(\widetilde gX)$ lie in the
closed half-space $\{z:\langle v,z\rangle\geq0\}$. On the other hand,
$p_1,\dots,p_d$ occur among these partial sums, so
$0\in\Conv(p_1,\dots,p_d)\subseteq H(\widetilde gX)$. Hence $v^{\perp}$
is a supporting hyperplane of $H(\widetilde gX)$ containing $0$, which
implies $0\in\partial H(\widetilde gX)$.

Conversely, suppose that $X\notin\mathcal D$. Then, for every
$g\in\mathfrak B_n$, the partial sums
$s_1(gX),\dots,s_n(gX)$ are in general linear position. Hence, by
Lemma~\ref{lem:separation}(iii), the origin is either outside $H(gX)$
or belongs to its interior. Thus
$0\notin\partial H(gX)$ for every $g\in\mathfrak B_n$, which proves
the converse by contraposition.
\end{proof}

\section{Absorption by convex hulls of  random bridges} \label{sec:bridge}

In this short section, we indicate how the method developed above can be adapted to random bridges. Since the argument closely parallels the random-walk case, we focus on the necessary modifications and omit details that are entirely analogous. The following theorem was proved in~\cite{KVZ17_GAFA}.

\begin{theorem}[Absorption probability for a random bridge]
\label{thm:bridge_absorption_for_iid}
Let $d\geq 1$,  $n\geq d+2$ and let $\xi_1,\dots,\xi_n$ be (in general, dependent) random vectors in $\R^d$ with partial sums
$$
S_i = \xi_1 + \dots + \xi_i,\quad  1\leq i\leq n,\quad  S_0=0.
$$
We impose the following assumptions on the increments $\xi_1,\dots,\xi_n$:
\begin{itemize}
\item[(i)] \emph{Bridge property:} $S_n=\xi_1+\dots+\xi_n = 0$ a.s.
\item[(ii)] \emph{Exchangeability:} For every permutation $\sigma$ of the set $\{1,\dots,n\}$, we have the distributional equality
$$
(\xi_{\sigma(1)},\dots,  \xi_{\sigma(n)}) \stackrel{d}{=} (\xi_1,\dots,\xi_n).
$$
\item[(iii)] \emph{General position of partial sums:}
For every $1\leq i_1 < \dots < i_d \leq n-1$, the probability that the vectors $S_{i_1}, \dots, S_{i_d}$ are linearly dependent is $0$.
\end{itemize}
Then
\begin{align}
    \pbridge{n}{d}
    &:=
    \pr[0 \in \Conv(S_1,\dots,S_{n-1})]
    =
    \frac{2}{n!}
    \sum_{r\geq 0}
    \stirling{n}{d+2+2r},
    \label{eq:def-p-br}
    \\
    \qbridge{n}{d}
    &:=
    \pr[0 \notin \Conv(S_1,\dots,S_{n-1})]
    =
    \frac{2}{n!}
    \sum_{r\geq 0}
    \stirling{n}{d-2r}.
    \label{eq:def-q-br}
\end{align}
Here, $\stirling{n}{k}$ with $n\in \{1,2,\dots\}$ and $k\in \mathbb Z$ denotes the unsigned Stirling number of the
first kind defined by the generating function
\begin{equation}\label{eq:stirling_generating_func}
\sum_{k\in \mathbb Z} \stirling nk x^k = x(x+1)(x+2)\dots (x+n-1).
\end{equation}
In particular, $\stirling nk= 0$ for $n\in \{1,2,\dots\}$ and $k\notin\{0,\dots,n\}$, so that the sums appearing above terminate after finitely many nonzero terms.
\end{theorem}

\smallskip
\emph{The configuration space, the group, and the walls.}
A bridge is determined by all of its increments except the
last one, so we take $\R^{d\times(n-1)}$ as the configuration space and attach to
$X=(x_1\,|\,\dots\,|\,x_{n-1})$ the bridge increments
$$
\widehat X:=\big(x_1\,|\,\dots\,|\,x_{n-1}\,|\,-(x_1+\dots+x_{n-1})\big),
$$
which is a linear bijection onto $\{Z\in\R^{d\times n}\mid s_n(Z)=0\}$. We keep
$s_k(X):=x_1+\dots+x_k$ for $1\leq k\leq n-1$. Put $s_0(X)=s_n(X)=0$.  In the bridge setting, define
$$
H(X):=\Conv(s_1(X),\dots,s_{n-1}(X)).
$$

The group is the symmetric group $\mathfrak{S}_n$.  It acts on $\R^{d\times(n-1)}$ by
restoring the last increment, permuting the columns, and deleting the last column
again. The walls are
$$
D_{\sigma,I}:=\big\{X \in \mathbb R^{d \times (n-1)} \mid \text{the vectors } s_i(\sigma X),\ i\in I,\ \text{are linearly dependent}\big\},
$$
where $\sigma \in \mathfrak{S}_n$,   $I\subseteq[n-1]$, $|I|=d$.
 We call
$X$ \emph{regular} if it lies on no wall. It suffices to prove the following result; averaging over $\mathfrak{S}_n$ turns it into Theorem~\ref{thm:bridge_absorption_for_iid}.

\begin{theorem}[Absorption for bridges: Combinatorial version]\label{thm:bridge_comb}
Let $d\geq 1$ and  $n\geq d+2$.
For every regular $X\in\R^{d\times(n-1)}$,
$$
\big|\{\sigma\in\mathfrak{S}_n \mid 0\notin H(\sigma X)\}\big|
=2\sum_{r\geq0}\stirling{n}{d-2r}.
$$
\end{theorem}

\smallskip
\emph{Crossing invariance.} Proposition~\ref{prop:pathexists_general} applies with $n$ replaced by $n-1$.
We can follow the argument of Section~\ref{sec:convex_hulls_random_walks}; only the construction of the involution in the proof of Proposition~\ref{prop:crossing} has to be
replaced. Let $\gamma : [-1,1] \to \R^{d\times (n-1)}$ be a piecewise linear path such that $\gamma(t)$ is regular for all $t \neq 0$. We consider a crossing furnished by
Proposition~\ref{prop:pathexists_general}, so that, for every
$\sigma\in\mathfrak S_n$, there is at most one $d$-element set $I$
with $\gamma(0)\in D_{\sigma,I}$. Let
$$
\mathcal M_\partial:=\{\sigma \in \mathfrak{S}_n \mid \, 0 \in \partial H(\sigma \gamma(0))\}.
$$
For $\sigma\in\mathcal M_\partial$ let $I(\sigma)=\{i_1<\dots<i_d\}\subseteq[n-1]$
index the vertices of the simplicial facet $F$ whose relative interior contains the
origin, and let $\nu$ be its unit normal, so that $\langle\nu,s_i(\sigma\gamma(0))\rangle=0$ for
$i\in I(\sigma)$ and $\langle\nu,s_i(\sigma\gamma(0))\rangle>0$ for the remaining $i\in[n-1]$; because
$s_0(\sigma\gamma(0))=s_n(\sigma\gamma(0))=0$, the first relation holds for $i_0:=0$ and $i_{d+1}:=n$ as well. Let
$\rho_{I(\sigma)}\in\mathfrak{S}_n$ \emph{reverse the order} of the $d+1$ blocks
$(x_{i_{k-1}+1},\dots,x_{i_k})$, leaving each block internally unchanged. Then $s_{\,n-i_j}(\rho_{I(\sigma)}\sigma\gamma(0))=-s_{i_j}(\sigma\gamma(0))$ for $j\leq d$ and, in particular,
\begin{align*}
\langle\nu,s_{n-i_j}(\rho_{I(\sigma)}\sigma\gamma(0))\rangle=0.
\end{align*}
Since the blocks remain internally unchanged and the sum of each block is orthogonal to $\nu$, we obtain
\begin{align*}
\langle\nu,s_{n-m}(\rho_{I(\sigma)}\sigma\gamma(0))\rangle > 0, \quad m \in [n-1]\setminus I(\sigma).
\end{align*}

Therefore
$\iota(\sigma):=\rho_{I(\sigma)}\sigma$ maps $\mathcal M_\partial$ into itself,
replacing the facet $F$ by $-F$ and $I(\sigma)$ by $\{n-i_d<\dots<n-i_1\}$, and it is an
involution because $\rho_{I(\iota(\sigma))}$ reverses the block order back. The rest of the proof of Proposition~\ref{prop:crossing} applies verbatim.

\smallskip
\emph{The value of the constant.}
As for the walk, we compute the constant of Theorem~\ref{thm:bridge_comb} on
one concrete configuration, which we choose lacunary in the spirit of
Proposition~\ref{prop:lacunary}. The increments of a bridge must sum  to
zero, so the last increment has to balance all the others: we take unit
vectors $u_1,\dots,u_{n-1}\in \R^d$ in general linear position and  consider
\begin{align*}
x_k&:=R^{k-1}u_k, \qquad 1\leq k\leq n-1,
\\
x_n&:=-(x_1+\dots+x_{n-1})=-R^{n-2}\big(u_{n-1}+O(R^{-1})\big),
\end{align*}
for some large $R$.

Fix $\sigma\in\mathfrak S_n$, write $s_m:=x_{\sigma(1)}+\dots+x_{\sigma(m)}$,
and let $b:=\sigma^{-1}(n)$ be the position of the balancing increment. For $m<b$ the summands are
lacunary, and the largest increment used so far dominates,
\[
s_m=R^{\mu-1}\bigl(u_\mu+O(R^{-1})\bigr),
\qquad \mu=\max\{\sigma(1),\dots,\sigma(m)\}.
\]

For $m\geq b$, since the increments sum  to zero, we can write $s_m=-\sum_{i>m}x_{\sigma(i)}$. This representation does not contain $x_n$, so all summands are lacunary and
\[
s_m=-R^{\nu-1}\bigl(u_\nu+O(R^{-1})\bigr),
\qquad \nu=\max\{\sigma(m+1),\dots,\sigma(n)\}.
\]
Before time $b$ the bridge is
governed by the largest step made so far, like a walk; after time $b$ it must
return to the origin and is governed by the largest step still to be
compensated.

As $m$ grows, $\mu$ runs through the record positions of the part of $\sigma$ before
$b$, and $\nu$ through the record positions of the part after $b$, read from the
right. We call these positions the \emph{two-sided record positions} of $\sigma$:
the positions of the entries visible either from the left or from the right end of the word
$(\sigma(1),\dots,\sigma(n))$, with the position of the maximal entry $n$ excluded. Every such entry
is visible from exactly one end, and we call it a $+$\,record or a
$-$\,record accordingly. Similarly to the walk model, for $R$ large, the origin belongs to the convex hull $\Conv(s_1, \dots, s_{n-1})$ if and only if
\[
0\in\Conv\big(\{u_{\sigma(i)} : i \text{ a $+$\,record}\}
\cup\{-u_{\sigma(i)} : i \text{ a $-$\,record}\}\big).
\]
The proof is identical to Cases~1 and~2 in
Proposition~\ref{prop:lacunary}, with the record vectors replaced by
the signed two-sided record vectors.

Now we count. Construct a permutation by inserting the values
$n,n-1,\dots,1$, each independently and uniformly into one of the slots of the current word. A
value becomes a two-sided record if and only if it falls into one of the two extreme
slots (the leftmost one or the rightmost one), and its sign is determined by the side, so the signs of the records
are independent fair coins. Thus, conditionally on the record values, the
signed record vectors are as in Wendel's theorem, and we arrive at Wendel's formula with a
random number of points:
\[
\pbridge{n}{d}
=\mathbb E\left[\pwend{K_n^{\mathrm{br}}}{d}\right]
=\sum_{k=1}^{n-1}\pr\big(K_n^{\mathrm{br}}=k\big)\,\pwend{k}{d},
\]
where $K_n^{\mathrm{br}}$ is the number of two-sided records of a uniform
random permutation of $[n]$.

The insertion construction also identifies the distribution of
$K_n^{\mathrm{br}}$: at each insertion after the first, when there are $m$ available slots,
the inserted value is a two-sided record with probability $2/m$, independently of the other values. So $K_n^{\mathrm{br}}$ has the same distribution as a sum of independent $\mathrm{Bern}(2/m)$-variables over $m=2,\dots n$.
This distribution arises if one reads a uniform permutation from left to
right and counts the entries that are running maxima or running minima, the first entry excluded. Explicitly,
\[
\pr\big(K_n^{\mathrm{br}}=k\big)=\frac{2^k}{n!}\stirling{n-1}{k},
\]
and substituting this together with~\eqref{eq:def-p-Wend} gives
\[
\pbridge{n}{d}
=\frac{2}{n!}\sum_{k=1}^{n-1}\stirling{n-1}{k}\sum_{j=d}^{k-1}\binom{k-1}{j}
=\frac{2}{n!}\sum_{r\geq0}\stirling{n}{d+2+2r}.
\]
The last identity is verified exactly as in
Lemma~\ref{lem:constants_equivalents}, with $x=1+t$ substituted into
$\sum_{k}\stirling{n-1}{k}x^k=x(x+1)\cdots(x+n-2)$. This is the formula of
Theorem~\ref{thm:bridge_absorption_for_iid} and, after multiplication by
$n!$ and taking the complement, the constant of Theorem~\ref{thm:bridge_comb}.

\medskip
\emph{Characterization of the exceptional set.}
The analogue of Proposition~\ref{prop:walk_exceptional_set} holds for
bridges: $X$ lies on some wall $D_{\sigma,I}$ if and only if for some
$\sigma\in\mathfrak S_n$ the origin lies on the boundary of the polytope
$$\Conv(s_1(\sigma X),\dots,s_{n-1}(\sigma X)).$$ For the nontrivial
implication, let $I=\{i_1<\dots<i_d\} \subseteq [n-1]$ be such that
$s_{i_1}(\sigma X),\dots,s_{i_d}(\sigma X)$ are linearly dependent. As in
the proof of Proposition~\ref{prop:walk_exceptional_set}, split the
increments into the $d+1$ blocks $\{i_{k-1}+1,\dots,i_k\}$, where $i_0:=0$
and $i_{d+1}:=n$, and let $z_1,\dots,z_{d+1}$ be the block sums, so that
$s_{i_k}(\sigma X)=z_1+\dots+z_k$. A linear relation among the $s_{i_k}(\sigma X)$
rewrites as $\sum_{k=1}^{d}\beta_k z_k=0$. In the bridge case, $z_1 + \dots +  z_{d+1} = 0$, so setting $\beta_{d+1} = 0$ we can write
\begin{align*}
    0=\sum_{k=1}^{d}\beta_k z_k
=
    \sum_{k=1}^{d+1}\beta_k z_k
+
    c\sum_{k=1}^{d+1} z_k
=
    \sum_{k=1}^{d+1}(\beta_k+c) z_k.
\end{align*}
Choosing $c = -\min_k \beta_k$ we ensure that all coefficients $\beta_k+c$ are nonnegative and at least one is zero.  Permuting the blocks so that the coefficients
decrease and summing by parts places the origin in the convex hull of the
partial sums at the block boundaries; rotating each block cyclically moves
all remaining partial sums to one side of a hyperplane through the origin,
similarly to the walk case.

\section{Deriving recurrence relations by insertion}\label{sec:insertion}

In this section, we treat all three models (Wendel, walks and bridges) in a unified way. In each model, we start with $n$ vectors in $\R^d$ (in the bridge case, assuming their sum is $0$) and form finitely many polytopes ($2^n$ in the Wendel case, $2^n n!$ in the walk case, and $n!$ in the bridge case). Assuming that the configuration is regular in the sense of the respective model, let $N_{n,d}^{\mathrm{Wend}}$, $N_{n,d}^{\mathrm{walk}}$, $N_{n,d}^{\mathrm{br}}$ be the number of polytopes that contain the origin. Our aim is to we prove recurrence relations for these numbers and use these to deduce probabilistic representations for the absorption probabilities.

\begin{theorem}[Recurrence relations] \label{thm:insertion_recurrences} For every $d\geq 1$, we have
\begin{align*}
N^{\mathrm{Wend}}_{n,d}&=N^{\mathrm{Wend}}_{n-1,d}+N^{\mathrm{Wend}}_{n-1,d-1}, \qquad n\geq d+1,\\
    N^{\mathrm{walk}}_{n,d}&=(2n-1)N^{\mathrm{walk}}_{n-1,d}+N^{\mathrm{walk}}_{n-1,d-1}, \qquad n\geq d+1,\\
    N^{\mathrm{br}}_{n,d}&=(n-1)N^{\mathrm{br}}_{n-1,d}+N^{\mathrm{br}}_{n-1,d-1}, \qquad n\geq d+2.
\end{align*}
For $n\geq1$, we use the boundary conditions
$N^{\mathrm{Wend}}_{n,0}=2^n$, $N^{\mathrm{Wend}}_{n,n} = 0$,  $N^{\mathrm{walk}}_{n,0}=2^n n!$,  $N^{\mathrm{walk}}_{n,n}=0$.
For $n\geq2$, we also set $N^{\mathrm{br}}_{n,0}=n!$,
$N^{\mathrm{br}}_{n,n-1}=0$.
\end{theorem}

The idea is to determine how the number of absorbing polytopes changes when a configuration of $n-1$ vectors in $\R^d$ is enlarged by adjoining a  ``small'' vector.
More precisely, consider a regular configuration $X=(x_1\,|\,\dots\,|\,x_{n-1})$ of $n-1$ vectors in $\mathbb{R}^d$. We append to $X$  a column $\delta u$ at position $n$, where $u\neq 0$ is a vector in $\R^d$ and $\delta > 0$ is small.   In the bridge case, we suppose that $x_1 + \cdots + x_{n-1} = 0$ and  $\delta u$ is at the same time subtracted from $x_1$, so that the new columns still sum to zero. Let $\overline{X}$ be the resulting enlarged configuration  of $n$ vectors in $\R^d$.

Now, each model has an associated symmetry group $G_n$, namely $G_n= \{\pm 1\}^n$ in the Wendel case, $G_n = \mathfrak{B}_n$ in the walk case, and $G_n = \mathfrak{S}_n$ in the bridge case. Acting on the enlarged configuration $\overline{X}$ by some element $\overline{g}\in G_n$ is the same as acting on $X$ by some element $g\in G_{n-1}$,  inserting the  vector $\eta \delta u$ at a suitable slot and choosing the sign $\eta \in \{\pm 1\}$ (in the bridge case, $\eta = 1$).

Let $T_n$ be the number of possible insertions.
\begin{itemize}
\item In the Wendel case, there are no permutations and the only choice is the sign of $\delta u$.
So  $T_n=2$.
\item
In the walk case, we insert $\delta u$ at one of $n$ possible slots and choose one of two possible signs. There are $T_n= 2n$ insertions.
\item
In the bridge case, we have $n$ slots and do not choose a sign. So $T_n=n$.
\end{itemize}
In all three cases, we have $T_n = |G_n|/|G_{n-1}|$. The recurrence relations appearing in Theorem~\ref{thm:insertion_recurrences} can be written in a unified way:
\begin{equation}\label{eq:recurrence_unified}
N_{n,d} = (T_n-1) N_{n-1,d} +  N_{n-1,d-1},
\end{equation}
where $N_{n,d}$ denotes $N_{n,d}^{\mathrm{Wend}}, N_{n,d}^{\mathrm{walk}}$ or $N_{n,d}^{\mathrm{br}}$ depending on the model.

Our aim is to prove this recurrence. The following lemma is the geometric core of the argument.
\begin{lemma}\label{lem:one_point}
Let $K\subset\R^{d}$ be a nonempty compact convex set and let $v\in \R^d$, $v\neq0$. Let $\pi$ be the orthogonal projection onto $v^{\perp}$. Then
\begin{equation}\label{eq:one_point}
    \mathbb{1}_{\{0\in\Conv(K\cup\{v\})\}}
   +\mathbb{1}_{\{0\in\Conv(K\cup\{-v\})\}}
   =\mathbb{1}_{\{0\in K\}}+\mathbb{1}_{\{0\in\pi K\}} .
\end{equation}
 If additionally $0\notin\partial K$ and $0\notin\partial_{v^{\perp}}(\pi K)$, where the latter boundary is taken in $v^{\perp}$, then
\begin{equation}\label{eq:one_point_ii}
    \mathbb{1}_{\{0\in\Conv(K_{+}\cup\{v\})\}}
   +\mathbb{1}_{\{0\in\Conv(K_{-}\cup\{-v \})\}}
   =\mathbb{1}_{\{0\in K\}}+\mathbb{1}_{\{0\in\pi K\}},
\end{equation}
whenever $K_{+}$ and $K_{-}$ are nonempty compact convex sets in $\R^d$
whose Hausdorff distance to $K$ is smaller than $\eps$, where $\eps>0$ is a sufficiently small number  depending only on $K$ and $v/\|v\|$.
\end{lemma}

\begin{proof}
Consider the rays $R_{\pm}=\{\pm sv:s\geq0\}$ and the line $R=R_{+}\cup R_{-}$ spanned by $v$.
For a nonempty compact convex set $A$ and a vector $w\neq0$ we have $0\in\Conv(A\cup\{w\})$ if
and only if $A$ meets the ray $\{-sw:s\geq0\}$. Hence the two indicators on the left-hand side of \eqref{eq:one_point} equal $\mathbb{1}_{\{K\cap R_{-}\neq\emptyset\}}$ and
$\mathbb{1}_{\{K\cap R_{+}\neq\emptyset\}}$. Since $K$ meets $R$ if and only if
$0\in\pi K$, and meets both rays if and only if $0\in K$, \eqref{eq:one_point} follows.

By the above observation with $A = K_{\pm}$, the left-hand side of \eqref{eq:one_point_ii} equals
\[
\mathbb{1}_{\{K_{-}\cap R_{+}\neq\emptyset\}}
+
\mathbb{1}_{\{K_{+}\cap R_{-}\neq\emptyset\}}.
\]
We shall need the following Hausdorff stability property:
If the origin is an interior point of a
compact convex set, then it is an interior point of every compact convex set sufficiently
close to it in the Hausdorff distance.

Since $0\notin\partial K$, there are three cases. If $0\in\operatorname{int}K$, then $0\in K_{\pm}$ by Hausdorff stability, and both sides of \eqref{eq:one_point_ii} equal $2$. If $0\notin K$
and $K$ misses $R$, then $K$ is at a positive distance from $R$, so $K_{\pm}$ miss $R$ as
well. In this case, both sides of \eqref{eq:one_point_ii} vanish.

Suppose finally that  $0\notin K$ and $K$ meets exactly one of the rays, without loss of generality
$R_{-}$. Then the right-hand side of \eqref{eq:one_point_ii} equals $1$.  Since $K$ is at a
positive distance from $R_{+}$, both $K_{+}$ and $K_{-}$ miss $R_{+}$; in particular $\mathbb{1}_{\{K_{-}\cap R_{+}\neq\emptyset\}}=0$.  Further, $0$ is an interior point of $\pi K$ in $v^{\perp}$,
so $0\in\pi K_{+}$ by Hausdorff stability, that is, $K_{+}$ meets $R$; since $K_{+}$ misses
$R_{+}$, it meets $R_{-}$, and $\mathbb{1}_{\{K_{+}\cap R_{-}\neq\emptyset\}}=1$. Hence the left-hand side of \eqref{eq:one_point_ii} also equals $1$.

The number $\eps>0$ should be chosen smaller than all relevant stability radii and positive distances. Since the objects involved, $R$, $R_+$, $R_-$ and $v^\perp$, depend on $v$ only through $v/\|v\|$, we can choose $\eps$ depending on $K$ and $v/\|v\|$ only.
\end{proof}

Recall that  $X=(x_1\,|\,\dots\,|\,x_{n-1})$ is a regular configuration of $n-1$  vectors in $\R^d$, and $\overline X$ is the enlarged configuration with appended $n$-th column $\delta u$. (In the bridge case, $\delta u $ is at the same time subtracted from $x_1$.)  For some element $\overline g\in G_n$, we consider $\overline g \overline X$ and represent it as the result of inserting of a column $\eta \delta u$ into $gX$ for a suitable $g\in G_{n-1}$.  Let $gX = (y_1 \, |\, \dots \,|\,  y_{n-1})$.
We write $K$ for the
convex hull attached to $gX$ by the model at hand:
\begin{align*}
    K =
    \begin{cases}
    \Conv (y_i : i \in [n-1]),\quad \text{ for the  Wendel model},\\
    \Conv (y_1+\dots+ y_i: i \in [n-1]),\quad \text{ for the random walk model},\\
    \Conv (y_1+\dots+ y_i: i \in [n-2]),\quad \text{ for the random bridge model.}
    \end{cases}
\end{align*}
Since $X$ is regular, $0\notin \partial K$. Similarly, let $\overline K$ be  the convex hull attached to $\overline g \overline X$.

In the next lemma we count insertions for which $0\in \overline K$.

\begin{lemma}[Insertions leading to absorption]\label{lem:insertion}
Let $u\in \R^d$, $u\neq 0$ and let $\pi: \R^d \to u^{\perp}$ be the orthogonal projection  onto the hyperplane $u^{\perp}$. Suppose that $0\notin \partial_{u^{\perp}}(\pi K)$, where the boundary is taken in $u^{\perp}$.
Then, for all sufficiently small $\delta>0$, the
number of admissible insertions of $\delta u$ (including the sign choice in the Wendel and walk models) for which $0\in \overline K$ equals
\[
    (T_n-1)\,\mathbb{1}_{\{0\in K\}}+\mathbb{1}_{\{0\in\pi K\}}.
\]
\end{lemma}

\begin{proof} For the Wendel model, $\overline g \overline X$ differs from $gX$ by one column, either $\delta u$ or $-\delta u$, always appended at position $n$. So $\overline K = \Conv (K\cup \{\eta \delta u\})$ with $\eta = \pm 1$.  By Lemma~\ref{lem:one_point} with $v = \delta u$,
\[
    \mathbb{1}_{\{0\in\mathrm{Conv}(K\cup\{\delta u\})\}}
    +\mathbb{1}_{\{0\in\mathrm{Conv}(K\cup\{-\delta u\})\}}
    =\mathbb{1}_{\{0\in K\}}+\mathbb{1}_{\{0\in\pi K\}},
\]
which proves the claim since $T_n=2$ in the Wendel case.

\medskip
In the random walk model, $\overline g \overline X$ differs from $gX$ by either $\delta u$ or $-\delta u$ that can be inserted at any of $n$ slots. If $\eta \delta u$ is not inserted as the first column, the new convex hull $\overline K$ coincides with $K$ up to $O(\delta)$ in Hausdorff metric. Since $0 \notin \partial K$, either both $K$ and $\overline K$ contain the origin, or neither does. This contributes
\[
    (2n-2)\,\mathbb{1}_{\{0\in K\}}.
\]
If $\eta \delta u$ is inserted as the first column, the new convex hull is given by $\overline K = \Conv(\{\eta \delta u\} \cup (\eta \delta u+ K))$, where $\eta = \pm1$. By Lemma~\ref{lem:one_point} with $v= \delta u$ and $K_\eta = \eta \delta u + K$, this contributes
\[\mathbb{1}_{\{0\in K\}}+\mathbb{1}_{\{0\in\pi K\}}.
\]
In total, there are $(2n-1)\mathbb{1}_{\{0\in K\}}+\mathbb{1}_{\{0\in\pi K\}}$ insertions for which $0\in \overline K$. This completes the proof in the random walk case.

\medskip
In the random bridge model,  $\overline g \overline X$ differs from $gX$ by $\delta u$ that can be inserted at any of $n$ slots, and by subtracting $\delta u$ from one of the remaining columns. If the insertion occurs neither at the first nor at the $n$-th slot, the new convex hull $\overline K$ coincides with $K$ up to $O(\delta)$ in Hausdorff metric. Since $0 \notin \partial K$, either both $K$ and $\overline K$ contain the origin, or neither does. These $n-2$ slots contribute
\[
    (n-2)\,\mathbb{1}_{\{0\in K\}}.
\]
If $\delta u$ is inserted into the first slot, $\overline K$ can be represented as $\Conv (\{\delta u\} \cup K_+)$ for some convex polytope $K_+$ at Hausdorff distance $O(\delta)$ from $K$. If $\delta u$ is inserted into the last slot, then the last point visited by the bridge of length $n$ is $-\delta u$ and $\overline K$ can be represented as $\Conv (\{-\delta u\} \cup K_-)$ for some convex polytope $K_-$ at Hausdorff distance $O(\delta)$ from $K$. By Lemma~\ref{lem:one_point} with $v = \delta u$, these two slots contribute
\[
\mathbb{1}_{\{0\in K\}}+\mathbb{1}_{\{0\in\pi K\}}.
\]
In total, there are $(n-1)\mathbb{1}_{\{0\in K\}}+\mathbb{1}_{\{0\in\pi K\}}$ insertions for which $0\in \overline K$. This completes the proof in the random bridge case.
\end{proof}

We are now ready to prove~\eqref{eq:recurrence_unified}. In order to calculate $N_{n,d}$ we have to sum $\mathbb{1}_{\{0 \in \overline{K}\}}$ over all $\overline{g} \in G_n$. Equivalently, we can take a sum over $g \in G_{n-1}$ of the respective number of insertions for which $0\in \overline K$. Choosing $u$ generic and then $\delta$ small, we may apply Lemma~\ref{lem:insertion} to each term in this sum. So $N_{n,d}$ is obtained by summing    $(T_n-1)\,\mathbb{1}_{\{0\in K\}}+\mathbb{1}_{\{0\in\pi K\}}$ over $g\in G_{n-1}$. Summing $\mathbb{1}_{\{0\in K\}}$ over $g\in G_{n-1}$ gives exactly $N_{n-1,d}$, while summing $\mathbb{1}_{\{0\in\pi K\}}$ gives $N_{n-1,d-1}$. The proof of~\eqref{eq:recurrence_unified} is complete.

\smallskip
The recurrence relations from Theorem~\ref{thm:insertion_recurrences} admit the following probabilistic interpretation.

\begin{corollary}[Probabilistic representations for absorption probabilities]\label{cor:bernoulli}
Let $U_2,U_3,\dots$ be independent random variables with $U_m\sim\mathrm{Bern}(1/T_m)$, where
$T_m=2m$ for the walk, $T_m=m$ for the bridge and $T_m=2$ for Wendel. Then
\[    \pwalk{n}{d}=\mathbb{P}\Big(\sum_{m=2}^{n}U_m\geq d\Big),
    \qquad
\pbridge{n}{d}=\mathbb{P}\Big(\sum_{m=3}^{n}U_m\geq d\Big),
\]
the sum for the bridge starting at $m=3$, and
\[
\pwend{n}{d}=\mathbb{P}\Big(\sum_{m=2}^{n}U_m\geq d\Big)
    =\mathbb{P}\left(\mathrm{Bin}(n-1,\tfrac12)\geq d\right).
\]
\end{corollary}

Indeed, dividing \eqref{eq:recurrence_unified} by $|G_n| = T_n|G_{n-1}|$ turns it into a
recurrence for the absorption probabilities $p(n,d) = N_{n,d}/|G_n|$:
\begin{equation}\label{eq:recurrence_prob}
    p(n,d) = \Bigl(1 - \frac{1}{T_n}\Bigr) p(n-1,d) + \frac{1}{T_n}\, p(n-1,d-1) .
\end{equation}
The probability
$\mathbb{P}(U_{m_0} + \dots +  U_d \geq d)$ (with $m_0=2$ for Wendel and walks and $m_0=3$ for bridges) satisfies \eqref{eq:recurrence_prob} (condition
on $U_n$) and has the same initial values as $p(n,d)$: both equal $1$ for $d = 0$, and
vanish for $d \geq 1$ at the smallest value of $n$ ($n=d$ for Wendel and walks, or $n=d+1$ for the bridge),
where the sum is empty and the convex hull is a single point distinct from the origin.

In words: we run the construction backwards in time, removing the vectors one at a
time.  Passing from $m$ to $m-1$ vectors undoes one of the $T_m$ equally likely
insertions.  By Lemma~\ref{lem:insertion}, $T_m-1$ of them give $0\in\overline{K}$
exactly when $0\in K$, and a single one gives $0\in\overline{K}$ exactly when
$0\in\pi K$: undoing it projects the configuration, and the dimension drops by one.
Hence the dimension drops with probability $1/T_m$ at each step, independently of the
other steps, and $U_m$ is the indicator of the drop.  The unrolling stops either when
the dimension reaches zero, where absorption is automatic, or when the vectors are
exhausted, where the hull is a single point distinct from the origin and there is no
absorption.  Thus the origin is absorbed in $\R^d$ exactly when there are at least $d$
drops, that is, when $U_{m_0}+\dots+U_n\geq d$.

\section{Probabilistic representations and recurrences}
\label{sec:absorption_properties}

The purpose of this section is to derive probabilistic representations and recurrence relations for the absorption probabilities directly from their algebraic formulas appearing in  Theorems~\ref{thm:1}, \ref{thm:wendel}, and~\ref{thm:bridge_absorption_for_iid}.
The three arrays we are interested in are defined as follows:
\begin{align}
\pwalk{n}{d}
    &:=
    \frac{2}{2^{n}n!}
    \sum_{r\geq 0}
    \stirlingb{n}{d+1+2r},
    \qquad n\geq 1, \quad d\geq 0,
    \label{eq:def-p-walk-rep}
    \\
\pbridge{n}{d}
    &:=
    \frac{2}{n!}
    \sum_{r\geq 0}
    \stirling{n}{d+2+2r},
    \qquad n\geq 2, \quad d\geq 0,
    \label{eq:def-p-br-rep}
    \\
\pwend{n}{d}
    &:=
    \frac{1}{2^{n-1}}
    \sum_{j=d}^{n-1}
    \binom{n-1}{j},
    \qquad n\geq 1, \quad d\geq 0.
\label{eq:def-p-Wend-rep}
\end{align}
As before, $\stirling{n}{k}$ and $\stirlingb{n}{k}$ are understood to be
zero for $k\notin\{0,\dots,n\}$, and an empty sum is zero.
Evaluation of the corresponding generating polynomials at $1$ and $-1$
shows that
\begin{equation}
    \pwalk{n}{0}=\pwend{n}{0}=1 \quad (n\geq 1),
    \qquad
    \pbridge{n}{0}=1 \quad (n\geq 2).
    \label{eq:p-zero}
\end{equation}

\begin{proposition}[Random-walk absorption probabilities]
\label{prop:p-walk}
Let $U_2,U_3,\dots$ be independent random variables with
$
U_m\sim\operatorname{Bern}(\frac 1 {2m})$,
$m\geq 2$.
Then, for $n\geq 1$ and $d\geq 0$,
\begin{equation}
    \pwalk{n}{d}
    =\pr\left(\sum_{m=2}^{n}U_m\geq d\right).
    \label{eq:p-walk-bernoulli}
\end{equation}
Consequently, for $n\geq 2$ and $d\geq 1$,
\begin{equation}
    \pwalk{n}{d}
    =\frac {2n-1}{2n} \pwalk{n-1}{d}
    +\frac {1}{2n} \pwalk{n-1}{d-1}.
    \label{eq:p-walk-recurrence}
\end{equation}
Together with the boundary conditions $\pwalk{n}{0}=1$ for all  $n\geq 1$ and $\pwalk{1}{d}=0$  for all $d\geq 1$, this
recurrence uniquely determines the array.
\end{proposition}

\begin{proof}
For $n\in \{1,2,\dots\}$ write
\[
    A_n(t):=\prod_{m=2}^{n}(t+2m-1)
    =\sum_{k=0}^{n-1}a_{n,k}t^k,
\]
where the product is empty for $n=1$ and $a_{n,k}=0$ for
$k\notin \{0,\dots,n-1\}$. Since
\[
    \sum_{k=0}^{n}\stirlingb{n}{k}t^k=(t+1)A_n(t),
\]
comparison of coefficients gives, for every $d\geq 0$,
\[
    \sum_{r\geq 0}\stirlingb{n}{d+1+2r}
    =\sum_{k=d}^{n-1}a_{n,k}.
\]
On the other hand,
\[
    \mathbb E\left[t^{\sum_{m=2}^{n}U_m}\right]
    =\prod_{m=2}^{n}\left(\frac {2m-1+t}{2m}\right)
    =\frac {A_n(t)}{2^{n-1}n!}.
\]
Thus~\eqref{eq:def-p-walk-rep} is the tail probability in
\eqref{eq:p-walk-bernoulli}.
Conditioning on $U_n$ yields~\eqref{eq:p-walk-recurrence}. The uniqueness statement is immediate.
\end{proof}

\begin{proposition}[Random-bridge absorption probabilities]
\label{prop:p-br}
Let $V_3,V_4,\dots$ be independent random variables with $V_m\sim\operatorname{Bern}(\frac{1}{m})$, $m\geq 3$.
Then, for $n\geq 2$ and $d\geq 0$,
\begin{equation}
    \pbridge{n}{d}
    =\pr\left(\sum_{m=3}^{n}V_m\geq d\right).
    \label{eq:p-br-bernoulli}
\end{equation}
Consequently, for $n\geq 3$ and $d\geq 1$,
\begin{equation}
    \pbridge{n}{d}
    =\frac{n-1}{n}\pbridge{n-1}{d}
    +\frac{1}{n}\pbridge{n-1}{d-1}.
    \label{eq:p-br-recurrence}
\end{equation}
Together with the boundary conditions  $\pbridge{n}{0}=1$  for all $n\geq 2$ and  $\pbridge{2}{d}=0$  for all $d\geq 1$, this
recurrence uniquely determines the array.
\end{proposition}

\begin{proof}
For $n\in \{2,3,\dots\}$ put
\[
    C_n(t):=\prod_{m=3}^{n}(t+m-1)
    =\sum_{k=0}^{n-2}c_{n,k}t^k,
\]
where the product is empty for $n=2$ and $c_{n,k}=0$ for
$k\notin \{0,\dots,n-2\}$. The ordinary Stirling generating function gives
\[
    \sum_{k=0}^{n}\stirling{n}{k}t^k=t(t+1)C_n(t),
\]
and hence
\[
    \sum_{r\geq 0}\stirling{n}{d+2+2r}
    =\sum_{k=d}^{n-2}c_{n,k}.
\]
Moreover,
\[
    \mathbb E\left[t^{\sum_{m=3}^{n}V_m}\right]
    =\prod_{m=3}^{n}\left(\frac{m-1+t}{m}\right)
    =\frac{2C_n(t)}{n!}.
\]
This proves~\eqref{eq:p-br-bernoulli}. Conditioning on $V_n$ gives~\eqref{eq:p-br-recurrence}. The uniqueness statement is immediate.
\end{proof}

The Bernoulli representations also connect the random-walk and
random-bridge absorption probabilities to Wendel's formula.
Recall that $K_n$ denotes the number of cycles of a uniformly
distributed random permutation of $\{1,\dots,n\}$,  so that
\[
    \pr(K_n=k)=\frac{1}{n!}\stirling{n}{k},
    \qquad 1\leq k\leq n.
\]
Further, let $K_{n}^{\mathrm{br}}$ be the number of cycles of an Ewens$(2)$
random permutation of $\{1,\dots,n-1\}$. Its distribution is
\[
    \pr\bigl(K_{n}^{\mathrm{br}}=k\bigr)
    =\frac{2^k}{n!}\stirling{n-1}{k},
    \qquad 1\leq k\leq n-1.
\]
The normalization in the last display follows by evaluating the ordinary
Stirling generating polynomial at $2$.

\begin{corollary}[Wendel--Stirling mixtures]
For $n\geq 1$ and $d\geq 0$,
\begin{equation}
    \pwalk{n}{d}
    =\mathbb E\left[\pwend{K_n}{d}\right]
    =\frac{1}{n!}\sum_{k=d+1}^{n}
        \stirling{n}{k}\pwend{k}{d}.
    \label{eq:p-walk-Wendel-mixture}
\end{equation}
For $n\geq 2$ and $d\geq 0$,
\begin{equation}
    \pbridge{n}{d}
    =\mathbb E\left[\pwend{K_{n}^{\mathrm{br}}}{d}\right]
    =\frac{1}{n!}\sum_{k=d+1}^{n-1}
        2^k\stirling{n-1}{k}\pwend{k}{d}.
    \label{eq:p-br-Wendel-mixture}
\end{equation}
\end{corollary}

\begin{proof}
In a uniform random permutation, distinguish the cycle containing $1$
and mark each remaining cycle independently with probability $1/2$.
Conditional on $K_n=k$, the number $M_n$ of marked cycles has distribution
$\operatorname{Bin}(k-1,1/2)$, and therefore
\[
    \pr(M_n\geq d\mid K_n=k)=\pwend{k}{d}.
\]
In the standard sequential construction of a uniform permutation, the
indicators that $m$ starts a new cycle are independent
$\operatorname{Bern}(1/m)$ random variables. After marking, the indicators
for $m\geq 2$ are independent $\operatorname{Bern}(1/(2m))$ random
variables. Proposition~\ref{prop:p-walk} now gives
\eqref{eq:p-walk-Wendel-mixture}.

For an Ewens$(2)$ permutation, the new-cycle probability at step $i$ is
$2/(i+1)$. Marking every cycle other than the one containing $1$ with
probability $1/2$ therefore produces independent indicators with success
probabilities $1/(i+1)$, $2\leq i\leq n-1$. Their sum has the same law as
$\sum_{m=3}^{n}V_m$. Conditional on $K_{n}^{\mathrm{br}}=k$, the number of marked
cycles is again $\operatorname{Bin}(k-1,1/2)$. Proposition~\ref{prop:p-br}
therefore yields~\eqref{eq:p-br-Wendel-mixture}.
\end{proof}

\appendix
\section{Convex geometry} \label{convex_geometry}

In this appendix, we collect some standard facts from convex geometry. We say that vectors $q_1,\dots,q_N\in\R^d$
are in general linear position if $(q_i)_{i\in I}$ is linearly
independent for every $I\subseteq[N]$ with $1\leq |I|\leq d$.

\begin{lemma} \label{lem:separation}
    Let $N,d\geq1$ and  $q_1, \dots, q_N \in \R^d$.
    \begin{itemize}
        \item[(i)] $0 \notin \Conv(q_1, \dots, q_N)$ if and only if
        there is $v \in \R^d$ with $\langle q_m, v \rangle > 0$ for all
        $m$.
        \item[(ii)] $0 \in \Int \Conv(q_1, \dots, q_N)$ if and only if
        there is \emph{no} $v \neq 0$ with $\langle q_m, v \rangle \ge
        0$ for all $m$.
        \item[(iii)] If $q_1, \dots, q_N$ are in general linear position, then
        either $0 \notin \Conv(q_1, \dots, q_N)$ or $0 \in \Int
        \Conv(q_1, \dots, q_N)$.
        \item[(iv)] Let $N \ge d+1$ and $0 \in \partial \Conv(q_1, \dots, q_N)$.
        If there is exactly one $d$-element set $I\subseteq [N]$ such that
        the vectors $(q_i)_{i \in I}$ are linearly dependent, then the points $q_i$, $i \in I$,
        are affinely independent,
        $0 \in \relint \Conv(q_i : i \in I)$, and there is a unique unit
        vector $\nu$ such that $$\langle q_i, \nu \rangle = 0 \text{ for } i \in I
        \quad
        \text{and}
        \quad
        \langle q_m, \nu \rangle > 0 \text{ for }\, m \notin I.$$

    \end{itemize}
\end{lemma}

\begin{proof}
Parts (i) and (ii) are standard consequences of the separating
hyperplane theorem. For (iii), suppose
$$0 \in \Conv(q_1,\dots,q_N) \setminus \Int\Conv(q_1,\dots,q_N).$$ By
(ii) there is $v \neq 0$ with $\langle q_m, v\rangle \ge 0$ for all
$m$; write $0 = \sum_m \alpha_m q_m$ with $\alpha_m \ge 0$,
$\sum_m \alpha_m = 1$. Pairing with $v$ forces $\langle q_m, v\rangle =
0$ for every $m$ in the support $S := \{m \mid \alpha_m > 0\}$, so the
vectors $(q_m)_{m \in S}$ lie in the hyperplane $v^\perp$. If
$|S| \le d$ they are linearly independent by general linear position,
contradicting the nontrivial relation $\sum_{m \in S} \alpha_m q_m =
0$; if $|S| > d$, any $d$ of them lie in the $(d-1)$-dimensional space
$v^\perp$ and are dependent, again a contradiction.

For (iv), write $0 = \sum_m \alpha_m q_m$ with $\alpha_m \ge 0$,
$\sum_m \alpha_m = 1$. By (ii) there is $v \neq 0$ with
$\langle q_m, v \rangle \ge 0$ for all $m$; we may assume $|v| = 1$. Pairing
the relation $0 = \sum_m \alpha_m q_m$ with $v$ shows that the set
$Z := \{m \mid \langle q_m, v \rangle = 0\}$ contains the support
$S := \{m \mid \alpha_m > 0\}$, which is nonempty since the $\alpha_m$ sum to
one.

The vectors $(q_m)_{m \in Z}$ lie in the hyperplane $v^\perp$, so any $d$ of
them are linearly dependent; hence $|Z| \le d$, for otherwise $Z$ would
contain two distinct $d$-sets of linearly dependent vectors, contradicting
the uniqueness of $I$. On the other hand, the vectors $(q_m)_{m \in S}$ are
linearly dependent as well, and if $|S| \le d-1$, then every $d$-set
containing $S$ would inherit this dependence; since $N \ge d+1$, there are
at least two distinct such $d$-sets, again contradicting the uniqueness of
$I$. Hence $d \le |S| \le |Z| \le d$, so $S = Z$ is a $d$-element set whose
vectors are linearly dependent, and thus $S = Z = I$. This proves the claims
about $\nu := v$; moreover, $0 = \sum_{i \in I} \alpha_i q_i$ with all
$\alpha_i > 0$.

If the points $q_i$, $i \in I$, were affinely dependent, their affine hull
--- a linear subspace, as it contains the origin --- would have dimension at
most $d-2$; then $d-1$ of the vectors $(q_i)_{i \in I}$ together with any
$q_m$, $m \notin I$, would form a second $d$-set of linearly dependent
vectors. So these points are affinely independent, and the representation of
the origin with strictly positive coefficients $\alpha_i$ means precisely
that $0 \in \relint \Conv(q_i : i \in I)$. Finally, the affine hull of
these points is a linear hyperplane, necessarily equal to $v^\perp$; a unit
vector as in (iv) is orthogonal to it, so it equals $\pm v$, and the sign is
fixed by the strict inequalities. This proves the uniqueness of $\nu$.
\end{proof}

\begin{lemma}[Stability of absorption]\label{lem:stability}
Let $q_1, \dots, q_N \in \R^d$ be such that
$0 \notin \partial \Conv(q_1, \dots, q_N)$. Then there is $\delta > 0$
such that every tuple $q'_1, \dots, q'_N \in \R^d$ with
$|q'_m - q_m| < \delta$ for all $m$ satisfies
$$
0 \in \Conv(q'_1, \dots, q'_N)
\iff
0 \in \Conv(q_1, \dots, q_N).
$$
\end{lemma}

\begin{proof}
If $0 \notin \Conv(q_1, \dots, q_N)$, then
Lemma~\ref{lem:separation}(i) provides $v \in \R^d$ with
$\langle q_m, v \rangle > 0$ for all $m$; these finitely many strict
inequalities persist under small perturbations of the tuple, so
$0 \notin \Conv(q'_1, \dots, q'_N)$, again by
Lemma~\ref{lem:separation}(i).

Otherwise $0 \in \Int \Conv(q_1, \dots, q_N)$. Consider
$$
\phi(q_1, \dots, q_N) :=
\max_{v \in \R^d,\ |v| = 1}\ \min_{1 \le m \le N}\
\langle q_m, v \rangle,
$$
a continuous function of the tuple, being a maximum of continuous
functions over the compact unit sphere. By
Lemma~\ref{lem:separation}(ii), $0 \in \Int \Conv(q_1, \dots, q_N)$ if
and only if $\phi(q_1, \dots, q_N) < 0$, and this condition persists
under small perturbations as well.
\end{proof}

\section{Existence of the path}\label{sec:paths_existence}
In this appendix we collect several lemmas about existence of paths in Euclidean space. The path is required to connect two given points and is allowed to hit certain codimension $1$ sets, called \emph{walls},  finitely many times, but is not allowed to hit intersections of two walls.

\subsection{Simplest setting: walls are hyperplanes}
The following result is needed in the proof of Sparre Andersen's formula in Section~\ref{sec:Sparre-Andersen}.

\begin{lemma}\label{lem:path_for_sparre_andersen}
Let $\mathcal L$ be a finite family of pairwise distinct affine hyperplanes in
$\mathbb R^n$, and let $x,y\in\mathbb R^n\setminus\bigcup_{H\in\mathcal L}H$.
Then there exists a piecewise linear path $\gamma\colon[-1,1]\to\mathbb R^n$ with
$\gamma(-1)=x$ and $\gamma(1)=y$ which meets $\bigcup_{H\in\mathcal L}H$ at
finitely many times $t_1 < \dots< t_L$ only,  and such that $\gamma(t_i)$
lies on exactly one hyperplane of $\mathcal L$ for every $i\in\{1,\dots,L\}$.
\end{lemma}

\begin{proof}
Let $\mathcal F$ be the (finite) family of all nonempty intersections $H\cap H'$
with $H,H'\in\mathcal L$, $H\neq H'$; each $F\in\mathcal F$ is an affine subspace
of dimension $n-2$. Since $x\notin F$ for every $F\in\mathcal F$, the
affine hull $\operatorname{aff}(\{x\}\cup F)$ is an affine hyperplane, and similarly for
$y$. Hence
\[
  E:=\bigcup_{H\in\mathcal L}H\;\cup\;
  \bigcup_{F\in\mathcal F}\bigl(\operatorname{aff}(\{x\}\cup F)\cup
  \operatorname{aff}(\{y\}\cup F)\bigr)
\]
is a finite union of hyperplanes and therefore a Lebesgue null set. Fix
$z\in\mathbb R^n\setminus E$ and let $\gamma$ be a piecewise linear path which traverses $[x,z]$ and then $[z,y]$.

Fix $H\in\mathcal L$. A segment meeting the affine set $H$ in two distinct points
is contained in $H$; since $x,y,z\notin H$, each of $[x,z]$ and $[z,y]$ meets $H$
in at most one point. As $\mathcal L$ is finite, $\gamma$ meets
$\bigcup_{H\in\mathcal L}H$ at finitely many parameters, none of which is $-1$,
$0$ or $1$.

Suppose $p=\gamma(t)$ lies on two distinct $H,H'\in\mathcal L$, say
$p\in[x,z]$. Then $F:=H\cap H'\in\mathcal F$ and $p\neq x$, so that
$z\in\operatorname{aff}\{x,p\}\subseteq\operatorname{aff}(\{x\}\cup F)\subseteq E$,
contradicting the choice of $z$; the case $p\in[z,y]$ is symmetric. Thus every
point at which $\gamma$ meets $\bigcup_{H\in\mathcal L}H$ lies on exactly one
hyperplane of $\mathcal L$, and listing the corresponding parameters in
increasing order gives the assertion.
\end{proof}
\subsection{Walls as determinantal varieties}
The following result is needed in the proof of Wendel's formula, Section~\ref{sec:wendel}. Let $d,n\geq 1$. The space of configurations is now $\R^{d\times n}$.  For a $d$-element
set $I = \{i_1,\dots,i_d\} \subseteq [n]$ define the \emph{wall}
$$
D_I := \{X = (x_1,\dots,x_n) \in \R^{d\times n} \mid
\text{the vectors } x_{i_1},\dots,x_{i_d} \text{ are linearly dependent}\}.
$$
As in Section~\ref{sec:wendel}, we call a configuration $X \in \R^{d\times n}$ \emph{regular}
if it belongs to no wall, that is, if every $d$ of its columns are linearly
independent.

\begin{proposition} \label{prop:pathexist_for_wendel}
Let $X, Y \in \R^{d\times n}$ be regular configurations. Then there exists a
piecewise linear path $\gamma\colon [-1,1] \to \R^{d\times n}$ with
$\gamma(-1) = X$ and $\gamma(1) = Y$ which meets $\bigcup_{|I|=d} D_I$ at
finitely many times $t_1 < \dots < t_L$ only, and such that $\gamma(t_\ell)$
lies on exactly one wall $D_I$ for every $\ell \in \{1,\dots,L\}$.
\end{proposition}

\begin{proof}
If $n < d$, there are no walls and the straight segment from $X$ to $Y$ has the
required properties, so we assume $n \geq d$.

\medskip
\emph{Step 1: $X$ and $Y$ differ in one column.} Suppose first that $X$ and $Y$
differ in the $j$-th column only. Denote the $j$-th columns of $X$ and $Y$ by
$a$ and $b$, respectively, and let $w_i$, $i \neq j$, be the common
remaining columns. We keep these fixed and move the $j$-th column from $a$ to
$b$ as follows.

For every $(d-1)$-element set $K \subseteq [n] \setminus \{j\}$ put
$H_K := \operatorname{span}\{w_i \mid i \in K\}$. Each $H_K$ is a linear
hyperplane in $\R^d$: the $d-1$ vectors spanning it are columns of the regular
configuration $X$ and hence linearly independent. The hyperplanes $H_K$ are
pairwise distinct: if $H_K = H_{K'}$ for some $K \neq K'$, then all the columns of $X$ indexed by $K\cup K'$, of which
there are at least $d$, would lie in a $(d-1)$-dimensional
subspace, so some $d$ of them would be linearly dependent, contradicting the
regularity of $X$. Finally, $a \notin H_K$ for every $K$, since otherwise the
$d$ columns of $X$ indexed by $\{j\} \cup K$ would be linearly dependent;
for the same reason, $b \notin H_K$.

By Lemma~\ref{lem:path_for_sparre_andersen} applied to the family
$\{H_K\}$, there is a piecewise linear path
$\eta\colon [-1,1] \to \R^d$ from $a$ to $b$ which meets $\bigcup_K H_K$ at
finitely many times only, and at each of these times $\eta(t)$ lies on exactly
one hyperplane $H_K$. Let $\gamma(t)$ be the configuration with $j$-th column
$\eta(t)$ and the other columns $w_i$; this is a piecewise linear path from
$X$ to $Y$. Any $d$ columns of $\gamma(t)$ not containing the $j$-th one are
columns of $X$ and hence linearly independent for all $t$, while the $d$
columns indexed by $\{j\} \cup K$ are linearly dependent if and only if
$\eta(t) \in H_K$. In other words, $\gamma(t) \in D_{\{j\} \cup K}$ if and only
if $\eta(t) \in H_K$. Hence $\gamma$ meets the walls only at the finitely many
times at which $\eta$ meets $\bigcup_K H_K$, and at each such time $\gamma(t)$
lies on exactly one wall.

\medskip
\emph{Step 2: the general case.} It suffices to connect $X$ and $Y$ by a finite
chain of regular configurations in which any two consecutive ones differ in
exactly one column: applying Step 1 to each consecutive pair, concatenating the
resulting paths and reparametrizing over $[-1,1]$ then yields the required
path $\gamma$.

We construct vectors $z_1,\dots,z_n \in \R^d$ inductively so that all
the configurations
\begin{align*}
P_k &:= (z_1,\dots,z_k,\, x_{k+1},\dots,x_n),
\qquad 0 \leq k \leq n,
\\
Q_k &:= (z_1,\dots,z_k,\, y_{k+1},\dots,y_n),
\qquad 0 \leq k \leq n,
\end{align*}
 are regular. This will finish the proof, since
$$
X = P_0,\ P_1,\ \dots,\ P_n = Q_n,\ Q_{n-1},\ \dots,\ Q_0 = Y
$$
is then a chain of the required form.

Suppose $z_1,\dots,z_{k-1}$ are already chosen so that $P_{k-1}$ and
$Q_{k-1}$ are regular (for $k=1$ this holds by assumption, since $P_0 = X$ and
$Q_0 = Y$). Consider all sets of $d-1$ columns of $P_{k-1}$ or of $Q_{k-1}$
with indices in $[n] \setminus \{k\}$. By regularity, each such set is linearly
independent and thus spans a linear hyperplane in $\R^d$. Since finitely many
hyperplanes cannot cover $\R^d$, we may choose $z_k$ lying on none of
them. Then $P_k$ is regular: any $d$ of its columns not containing $z_k$
occur already in $P_{k-1}$ and are linearly independent, while any $d$ columns
containing $z_k$ are linearly independent by the choice of $z_k$.
Likewise $Q_k$ is regular, which completes the induction.
\end{proof}

\subsection{Walls as transformed determinantal varieties}

The following result is needed in  Sections~\ref{sec:convex_hulls_random_walks} and~\ref{sec:bridge}.
Fix a finite set $\mathcal B \subset \operatorname{GL}_n(\R)$ of
invertible matrices.  The walls are now
indexed by pairs $(B,I)$, where  $B \in \mathcal B$ and
$I \subseteq [n]$ is a $d$-element set. The corresponding wall is defined as
$$
D_{B,I} := \{X \in \R^{d\times n} \mid \det((XB)_I) = 0\},
$$
where $(XB)_I$ denotes the $d\times d$ submatrix of $XB$ formed by the columns
indexed by $I$. As before, we call $X \in \R^{d\times n}$ \emph{regular} if it
belongs to no wall. Proposition~\ref{prop:pathexist_for_wendel} is the special
case $\mathcal B = \{\mathrm{Id}\}$ of the next statement, but the path
constructed below is not made of coordinate moves: instead, we show that the
configurations lying on two walls with the same $B$ can be covered by finitely
many images of smooth maps in $dn-2$ variables --- an explicit substitute for
``codimension two'' --- so that the cone over this set with vertex at a regular
point is covered by images of maps in $dn-1$ variables and is therefore a
Lebesgue null set. Any regular point $Z$ outside two such cones then serves as
the middle vertex of a two-segment path.

\begin{proposition}[A polygonal path avoiding double degeneracies]\label{prop:pathexists_general}
Let $d, n$ be positive integers, let $\mathcal B \subset \operatorname{GL}_n(\R)$
be finite, and let $X, Y \in \R^{d\times n}$ be regular. Then there exists a
piecewise linear path $\gamma\colon [-1,1] \to \R^{d\times n}$ consisting of at
most two line segments such that $\gamma(-1) = X$, $\gamma(1) = Y$, the path
meets $\bigcup_{B \in \mathcal B} \bigcup_{|I|=d} D_{B,I}$ at finitely many
times $t_1 < \dots < t_L$ only, and for every $\ell \in \{1,\dots,L\}$ and
every $B \in \mathcal B$ the point $\gamma(t_\ell)$ lies on at most one wall
$D_{B,I}$.
\end{proposition}

The proof is given at the end of the subsection. It rests on three lemmas.

\begin{lemma}[Null images]\label{lem:null_images}
Let $m < N$ be positive integers, let $V \subseteq \R^m$ be open, and let
$\Psi\colon V \to \R^N$ be a $C^1$ map. Then $\Psi(V)$ is a Lebesgue null set
in $\R^N$; in particular, it has empty interior.
\end{lemma}

\begin{proof}
The open set $V$ is the union of the countably many dyadic cubes contained in
it, so it suffices to show that $\Psi(Q)$ is null for every closed cube
$Q \subset V$. On the compact convex set $Q$ the derivative of $\Psi$ is
bounded, so $\Psi$ is Lipschitz on $Q$ with some constant $L$. Subdivide $Q$,
of side length $s$, into $k^m$ congruent subcubes. Each subcube has diameter
$\sqrt m\, s/k$, so its image is contained in a ball of radius
$L \sqrt m\, s / k$. Hence $\Psi(Q)$ is covered by $k^m$ balls of total volume
at most $C k^{m-N}$, where $C$ does not depend on $k$; letting $k \to \infty$
and using $m < N$ shows that $\Psi(Q)$ is null. Finally, a null set has empty
interior, since a nonempty open set has positive measure.
\end{proof}

\begin{lemma}[Linearly dependent tuples]\label{lem:dependent_tuples}
Let $d \geq 2$. The set of linearly dependent tuples
$(a_1, \dots, a_{d-1}) \in (\R^d)^{d-1}$ is the union of $d-1$ images of
polynomial maps defined on $\R^{d(d-1)-2}$.
\end{lemma}

\begin{proof}
If the tuple is linearly dependent, then $\sum_i \lambda_i a_i = 0$ with
$\lambda_j \neq 0$ for some $j$, whence $a_j = \sum_{i \neq j} t_i a_i$ with
$t_i := -\lambda_i/\lambda_j$. The set in question is therefore covered by the
images of the $d-1$ polynomial maps
$$
\Phi_j\colon \big((a_i)_{i \neq j},\ (t_i)_{i \neq j}\big)
\longmapsto (a_1, \dots, a_{d-1}),
\qquad
a_j := \sum_{i \neq j} t_i a_i,
$$
defined on $(\R^d)^{d-2} \times \R^{d-2} \cong \R^{(d+1)(d-2)}$, and
$(d+1)(d-2) = d(d-1) - 2$.
\end{proof}

\begin{lemma}[Two distinct maximal minors]\label{lem:two_minors}
Let $1 \leq d < n$ and let $I, J \subseteq [n]$ be distinct $d$-element
sets. Then the set
$$
\{W \in \R^{d\times n} \mid \det W_I = \det W_J = 0\}
$$
is covered by finitely many images of $C^\infty$ maps defined on open subsets of
$\R^{dn-2}$. Here $W_I$ denotes the $d\times d$ submatrix of
$W$ formed by the columns indexed by $I$.
\end{lemma}

\begin{proof}
Suppose first that $d = 1$. Then $I = \{a\}$ and $J = \{b\}$ with $a \neq b$,
and the set is the linear subspace $\{W \mid w_a = w_b = 0\}$ of dimension
$n - 2$, the image of a linear embedding of $\R^{dn-2}$.

Now let $d \geq 2$. Since $I \neq J$ and $|I| = |J|$, we may choose
$a \in I \setminus J$ and $b \in J \setminus I$. Write $w_1, \dots, w_n \in \R^d$
for the columns of $W$. Expanding the determinants along the columns $w_a$ and
$w_b$ gives cofactor vectors $c_I(W), c_J(W) \in \R^d$ with
$$
\det W_I = \langle c_I(W), w_a\rangle,
\qquad
\det W_J = \langle c_J(W), w_b\rangle.
$$
The vector $c_I(W)$ depends only on the columns indexed by $I \setminus \{a\}$;
in particular, since $b \notin I$, it depends neither on $w_a$ nor on $w_b$,
and likewise for $c_J(W)$. Moreover, if the columns $(w_i)_{i \in I \setminus \{a\}}$
are linearly independent, then $c_I(W) \neq 0$: completing them by a vector $v$
to a basis of $\R^d$ yields $\langle c_I(W) , v \rangle \neq 0$, this being the
determinant of a basis.

Every $W$ with $\det W_I = \det W_J = 0$ falls into one of three cases, which
we parametrize separately.

\smallskip
\emph{Case 1: the columns $(w_i)_{i \in I \setminus \{a\}}$ are linearly
dependent.} By Lemma~\ref{lem:dependent_tuples}, these $d-1$ columns are
covered by $d-1$ images of polynomial maps in $d(d-1)-2$ variables; the
remaining $n-d+1$ columns of $W$ enter the parametrization as free variables.
This yields polynomial maps defined on
$\R^{d(d-1)-2} \times \R^{d(n-d+1)} = \R^{dn-2}$ whose images cover this case.

\smallskip
\emph{Case 2: the columns $(w_i)_{i \in J \setminus \{b\}}$ are linearly
dependent.} Symmetric to Case 1.

\smallskip
\emph{Case 3: both systems of columns are linearly independent.} Then
$c_I(W) \neq 0$ and $c_J(W) \neq 0$, so $(c_I(W))_r \neq 0$ and
$(c_J(W))_s \neq 0$ for some $r, s \in \{1,\dots,d\}$. Fix such a pair
$(r,s)$. Where these two coordinates do not vanish, the system
$\det W_I = \det W_J = 0$ is equivalent to
$$
(w_a)_r = -\frac{\sum_{\ell \neq r} (c_I(W))_\ell\, (w_a)_\ell}{(c_I(W))_r},
\qquad
(w_b)_s = -\frac{\sum_{\ell \neq s} (c_J(W))_\ell\, (w_b)_\ell}{(c_J(W))_s} .
$$
Since $c_I$ and $c_J$ do not depend on $w_a$ and $w_b$, the right-hand sides
involve neither $(w_a)_r$ nor $(w_b)_s$: they are rational functions of the
remaining $dn - 2$ entries of $W$, with nonvanishing denominators on the open
set
$$
V_{r,s} := \{(c_I)_r \neq 0,\ (c_J)_s \neq 0\} \subseteq \R^{dn-2}
$$
(both conditions involve only the retained entries). This part of the set is
thus the graph of a rational map over $V_{r,s}$, that is, the image of the
$C^\infty$ map $V_{r,s} \to \R^{d\times n}$ which copies the retained entries
and fills in the two solved ones. Letting $(r,s)$ range over the $d^2$
possible pairs completes the covering.
\end{proof}

\begin{proof}[Proof of Proposition~\ref{prop:pathexists_general}]
If $d \ge n$, the straight segment from $X$ to $Y$ has all the required properties: for $d > n$ there are no walls, while for $d = n$ the only $d$-element subset of $[n]$ is $[n]$ itself, so for every $B \in \mathcal B$ there is a single wall and the last condition holds automatically; moreover, each function $t \mapsto \det((\gamma(t)B)_{[n]})$ is a polynomial in one variable, not identically zero since its value at $t=-1$ is nonzero by regularity of $X$, so the walls are met at finitely many times only.

In the following we assume $d < n$.
Let
$$
\Sigma := \bigcup_{B \in \mathcal B}\ \bigcup_{I \neq J}
\big( D_{B,I} \cap D_{B,J} \big)
$$
denote the set of doubly degenerate configurations, the inner union being over
distinct $d$-element subsets $I, J \subseteq [n]$.

\medskip
\emph{Step 1: Covering}. We show that $\Sigma$ is covered by finitely many images of $C^\infty$ maps
defined on open subsets of $\R^{dn-2}$. The intersection
$D_{B,I} \cap D_{B,J}$ is the image of the set
$\{W \mid \det W_I = \det W_J = 0\}$ under the linear isomorphism
$W \mapsto WB^{-1}$ of $\R^{d\times n}$. If $\Psi\colon V \to \R^{d\times n}$
is one of the maps covering the latter set in Lemma~\ref{lem:two_minors}, then
$v \mapsto \Psi(v)B^{-1}$ is a $C^\infty$ map on the same domain, and the images of
these maps cover $D_{B,I} \cap D_{B,J}$. There are finitely many triples
$(B, I, J)$, so the claim follows.

\medskip
\emph{Step 2: Intermediate point}. We want to choose an intermediate point $Z$ such that the path traversing the segments $[X,Z]$ and $[Z,Y]$ does not meet $\Sigma$. We define two ``forbidden'' sets $C_X$ and $C_Y$ such that any $Z$ outside these sets satisfies this condition. For a point
$P \in \{X, Y\}$ consider the cone over $\Sigma$ with vertex $P$,
$$
C_P := \{P + \lambda(U - P) \mid U \in \Sigma,\ \lambda \geq 1\}.
$$
If $\Psi\colon V \to \R^{d\times n}$ is one of the maps from Step 1, defined on
an open $V \subseteq \R^{dn-2}$, then the part of $C_P$ generated
by $U \in \Psi(V)$ is contained in the image of the $C^1$ map
$$
(v, \lambda) \longmapsto P + \lambda\big(\Psi(v) - P\big),
$$
defined on the open set $V \times (0,\infty) \subseteq \R^{(dn-2)+1} = \R^{dn-1}$. By Lemma~\ref{lem:null_images}, $C_X \cup C_Y$ is a
Lebesgue null set and therefore has empty interior. The set of regular
configurations is open, being the complement of the finitely many zero sets of
the polynomials $W \mapsto \det((WB)_I)$, and nonempty, since it contains $X$.
Hence we may choose a regular configuration
$$
Z \in \R^{d\times n} \setminus (C_X \cup C_Y).
$$

\medskip
\emph{Step 3: the segments $[X,Z]$ and $[Z,Y]$ do not meet $\Sigma$.} Suppose
$U = X + t(Z - X) \in \Sigma$ for some $t \in [0,1]$. Since the regular point
$X$ lies on no wall, $t \neq 0$, and then
$Z = X + t^{-1}(U - X) \in C_X$, contrary to the choice of $Z$. Symmetrically,
a point $U = Y + t(Z - Y) \in \Sigma$ with $t \in (0,1]$ would give
$Z \in C_Y$.

\medskip
\emph{Step 4: conclusion.} Let $\gamma$ traverse $[X,Z]$ and then $[Z,Y]$,
reparametrized over $[-1,1]$. Fix $B \in \mathcal B$ and $I$. The restriction
of the polynomial $W \mapsto \det((WB)_I)$ to either segment is a polynomial
in one variable which is not identically zero, since its value at the endpoint
$X$ (respectively $Y$) is nonzero by regularity; hence it has finitely many
zeros. As there are finitely many pairs $(B, I)$, the path $\gamma$ meets the
walls at finitely many times only. Finally, at each such time
$\gamma(t_\ell) \notin \Sigma$ by Step 3; that is, for every $B \in \mathcal B$
at most one wall $D_{B,I}$ contains $\gamma(t_\ell)$.
\end{proof}

\section*{Acknowledgement}
ZK was supported by the DFG under Germany's Excellence Strategy EXC 2044 - 390685587, Mathematics M\"unster: \emph{Dynamics-Geometry-Structure}, by the DFG priority program SPP 2265 \emph{Random Geometric Systems}, and by the DFG Research Training Group \emph{Rigorous Analysis of
Complex Random Systems} (RTG 3027, Project Number 524444762). Part of this work was completed during a visit by A.\ Tarasov to
the University of M\"unster.

\section*{Declarations}

\subsection*{Statement on the use of generative AI}

The central ideas of this paper were developed by the authors several years
before the advent of generative AI.

Generative AI was used extensively during the subsequent development  of the paper. In particular, Claude  proposed a clearer
and substantially streamlined proof of the path-existence results in
Appendix~\ref{sec:paths_existence}, replacing an earlier incomplete argument. Claude also substantially
simplified and clarified several other proofs and suggested many essential ideas,
including the use of the moment curve in the proof of Wendel's theorem.
ChatGPT was used for mathematical and linguistic proofreading.

All suggestions produced by the AI were critically examined and
independently verified by the authors.

\subsection*{Conflict of interest statement}
The authors declare that they have no conflicts of interest.

\subsection*{Data availability statement}
We do not analyse or generate any datasets.

\bibliography{references}
\bibliographystyle{plainnat}

\end{document}